\documentclass[12pt, a4 paper]{amsart}
\usepackage{amssymb,amsmath,amsfonts,latexsym, tikz}
\usepackage{bm}
\usepackage{comment}
\usepackage{xcolor}
\newcommand{\newsection}[1]{\setcounter{equation}{0} \section{#1}}
\newcommand{\bea}{\begin{eqnarray}}
\newcommand{\eea}{\end{eqnarray}}

\newcommand{\clb}{\mathcal{B}}

\newcommand{\clh}{\mathcal{H}}
\newcommand{\clk}{\mathcal{K}}
\newcommand{\cll}{\mathcal{L}}
\newcommand{\clm}{\mathcal{M}}
\newcommand{\cln}{\mathcal{N}}

\newcommand{\cls}{\mathcal{S}}

\newcommand{\D}{\mathbb{D}}

\newcommand{\raro}{\rightarrow}

\def\textmatrix#1&#2\\#3&#4\\{\bigl({#1 \atop #3}\ {#2 \atop #4}\bigr)}
\def\dispmatrix#1&#2\\#3&#4\\{\left({#1 \atop #3}\ {#2 \atop #4}\right)}
\newcommand{\be}{\begin{equation}}
\newcommand{\ee}{\end{equation}}
\newcommand{\ben}{\begin{eqnarray*}}
\newcommand{\een}{\end{eqnarray*}}

\newcommand{\bi}{\begin{itemize}}
\newcommand{\ei}{\end{itemize}}

\newcommand{\T}{\mathbb{T}}

\theoremstyle{definition}

\newtheorem*{thm*}{Theorem}
\newtheorem*{quesn*}{Question}

\theoremstyle{plain}

\newtheorem{thm}{Theorem}[section]
\newtheorem{cor}[thm]{Corollary}
\newtheorem{lem}[thm]{Lemma}
\newtheorem{prop}[thm]{Proposition}
\theoremstyle{definition}
\newtheorem{defn}[thm]{Definition}
\newtheorem*{defn*}{Definition}

\newtheorem{ex}[thm]{Example}

\numberwithin{equation}{section}

\let\phi=\varphi

\begin{document}

\title[Finite-Rank Commutators]{Finite-Rank Commutators of projections onto shift-invariant subspaces}

\author[Biswas]{Rounak Biswas}
\address{Department of Mathematics, Indian Institute of Technology Palakkad, Kerala - 678623, India.}
\email{xyzrounak3@gmail.com}

\author[Sarkar] {Srijan Sarkar }
\address{Department of Mathematics, Indian Institute of Technology Palakkad, Kerala - 678623, India.}
\email{srijansarkar@gmail.com, srijans@iitpkd.ac.in}



\begin{abstract}
 In this article, using Halmos' two projections theorem, we completely characterize finite-rank commutators of orthogonal projections onto shift-invariant subspaces $\phi_1 H^2(\D)$ and $\phi_2 H^2(\D)$ of the Hardy space $H^2(\D)$ corresponding to inner functions $\phi_1, \phi_2$ on the unit disc $\D$. Using our methods, we connect the finite-rank commutator with the Fredholmness of the projections $(P_{\phi_1}, P_{\phi_2})$ as introduced by Avron, Seiler and Simon. We conclude with several characterizations on the polydisc.
\end{abstract}

\maketitle

\newsection{Introduction}\label{sec: 1}
The primary objective of this article is to classify pairs of orthogonal projections onto shift-invariant closed subspaces of Hardy spaces for which the commutator has finite--rank. Recall that an orthogonal projection $P$ on a Hilbert space $\clh$ is a bounded self-adjoint operator satisfying the idempotent condition, that is $P^2 = P$. A pair of orthogonal projections $(P, Q)$ on $\clh$ is said to be essentially commuting if the commutator $[P, Q]$ defined by 
\[
[P, Q]:= PQ - QP,
\]
is a compact operator. Essential commutativity of projections is an active topic of study \cite{ACL, ACG} and is related to several fundamental questions in operator theory. Our focus is on a class of projections that plays a fundamental role in the interplay between operator theory and complex function theory. More specifically, we study orthogonal projections onto invariant subspaces of $H^2(\D)$ (the Hardy space over the unit disc $\D$). Recall that 
\[
H^2(\D) := \{f \in \mathcal{O}(\D): \sup_{0 \leq r <1} \int_{\mathbb{T}} |f(rz)|^2 \, d\mu < \infty\},
\]
where $\mu$ is the normalized Lebesgue measure on $\mathbb{T}$ (the unit circle) and
\[
\|f\|: =  \big( \sup_{0 \leq r <1} \int_{\mathbb{T}} |f(rz)|^2 \, d\mu \big)^{\frac{1}{2}} \quad (f \in H^2(\D)).
\]
The collection of bounded analytic functions on $\D$ are contained inside $H^2(\D)$ and is denoted by $H^{\infty}(\D)$. It is a well-known phenomenon that $H^2(\D)$ is isometrically isomorphic to $H^2(\mathbb{T})$, and this correspondence gives rise to a huge collection of bounded operators coming from functions in $L^{\infty}(\T)$. More precisely, corresponding to a $\phi \in L^{\infty}(\T)$, the Toeplitz operator $T_{\phi}$ on $H^2(\D)$ is defined by $T_{\phi} f = P_{H^2(\D)} L_{\phi}|_{H^2(\D)}$, where $L_{\phi}$ is the multiplication operator on $L^2(\mathbb{T})$. These operators behave like a multiplication operator, when the symbols $\phi$ come from $H^{\infty}(\D)$, that is, $T_f g= fg$, when $f \in H^{\infty}(\D)$ and $g \in H^2(\D)$. One of the most simple yet profound property connecting the behaviour of functions with operator-theoretic properties of Toeplitz operators is the following: a Toeplitz operator $T_{\phi}$ is an isometry if and only if $\phi$ and is an \textit{inner} function \cite{BH}. Recall that inner functions $\phi$ are functions in $H^{\infty}(\D)$, satisfying $|\phi(e^{it})|=1$, almost everywhere on $\mathbb{T}$. Another important connection comes via the Toeplitz operator $T_z$ also known as the shift-operator. It was shown by Beurling \cite{Beurling}, that any $T_z$-invariant closed subspace of $\cls \subseteq H^2(\D)$ is of the form $\cls = \theta H^2(\D)$, where $\theta$ is an inner function. Using Toeplitz operators we can explicitly write down the orthogonal projection to any shift-invariant subspace. In particular, let $P_{\theta}$ denote the orthogonal projection onto the subspace $\theta H^2(\D)$, then $P_{\theta} = T_{\theta} T_{\theta}^*$ because $T_{\theta}$ is an isometry. Recently, in \cite[Theorem 2.2]{DPS}, Debnath et al. characterized when a pair of these projections commute. In particular, they proved that $[P_{\phi_1}, P_{\phi_2}]=0$ if and only if $\phi_1$ divides $\phi_2$ or $\phi_2$ divides $\phi_1$. It is an important instance, where operator-theoretic properties on pair of projections induces a factorization of inner functions. This motivates us to study, the following general question.

\begin{quesn*}
For which pairs of inner functions $\phi_1,\phi_2$ on $\D$, does the commutator $[P_{\phi_1},P_{\phi_2}]$ has finite--rank?
\end{quesn*}

Our approach to this question is based on the following: \begin{enumerate}
\item[(i)] finding representations of the orthogonal projections $P_{\phi_1}, P_{\phi_2}$ using Halmos' two projection theorem (see Theorem \ref{Halmos_two_subspace}),
\item[(ii)] simplifying the Halmos decomposition by factoring out the common inner divisor of $\phi_1$ and $\phi_2$, as this operation preserves the finite-rank property of the commutator $[P_{\phi_1}, P_{\phi_2}]$ (see Lemma \ref{gcd}).,
\item[(ii)] use Coburn's theorem for Toeplitz operators to simplify the decomposition further.
\end{enumerate}

Following the above procedure, we have completely characterized when the commutator $[P_{\phi_1}, P_{\phi_2}]$ becomes finite-rank. In Corollary \ref{without gcd1}, we prove that if $\phi=\mbox{gcd}(\phi_1,\phi_2)$, then
\[
[P_{\phi_1}, P_{\phi_2}] \text{ is of finite-rank if and only if } \tfrac{\phi_1}{\phi} \text{ or } \tfrac{\phi_2}{\phi} \text{ is a finite Blaschke product},
\]
and the rank of $
 [P_{\phi_1}, P_{\phi_2}] =2 \text{min}\{\text{deg}(\tfrac{\phi_1}{\phi}), \text{deg}(\tfrac{\phi_2}{\phi})\}
$. Here the degree of an inner function $\phi$ is defined to be the number of its zeros inside $\D$, counted with multiplicity, whenever $\phi$ is a finite Blaschke product; otherwise, we set $\deg(\phi)=\infty$. 

Interestingly, finite-rank commutators of these projections have a connection with the Fredholm property of pairs of projections as introduced by Avron, Seiler and Simon in \cite{ASS}. Following \cite{ASS}, a pair of orthogonal projections $(P,Q)$ on a separable Hilbert space is defined to be a Fredholm pair if the operator $QP$, viewed as a map from the range of $P$ to the range of $Q$ is a Fredholm operator. In Corollary \ref{lastfredhlom}, we prove that a pair of projections $(P_{\phi_1}, P_{\phi_2})$ with finite--rank commutator become a Fredholm pair if and only if and both $\tfrac{\phi_1}{\phi}$ and $\tfrac{\phi_2}{\phi}$ are finite Blaschke products. Moreover, in this case,
\[
\mbox{index}(P_{\phi_1},P_{\phi_2})=\text{deg}(\tfrac{\phi_2}{\phi})-\text{deg}(\tfrac{\phi_1}{\phi}).
\]

We have several interesting results in the polydisc case as well. The motivation comes from the recent result by Debnath et al. \cite{DPS}, characterizing the commutativity of the orothogonal projections $P_{\phi_1}$ and $P_{\phi_2}$ on $H^2(\D^n)$ (that is, the Hardy space over unit polydisc $\D^n$). In particular, they prove that $[P_{\phi_1}, P_{\phi_2}] = 0$ if and only if there exist inner functions $\psi, \tilde{\phi_1}, \tilde{\phi_2}$, where $\tilde{\phi_1}$ and $\tilde{\phi_2}$ depends on disjoint subset of variables in $\{z_1, \ldots, z_n\}$, such that 
\[
\phi_1 = \psi \, \tilde{\phi_1} \text{ and } \phi_2 = \psi \, \tilde{\phi_2}.
\]
This led Debnath et al. to completely answer (\cite[Theorem 3.2]{DPS}) the following question raised by R.G.Douglas: when does the product $(I_{H^2(\D^n)} - P_{\phi_1})(I_{H^2(\D^n)} - P_{\phi_2})$ become a finite-rank projection. In the recent article \cite{ZTLYZ}, Zhang et al. observed that the compactness of this product is a sufficient condition to conclude that it must be a finite-rank projection. Their method depends on the behaviour of joint defect operator $\Delta_T$ corresponding to pair of isometries $(T_1, T_2)$ as observed by Guo and Wang in \cite{GW}. In Theorem \ref{compact-proj1} and Theorem \ref{compact-proj2}, we give new proofs of the observations by Zhang et al. using our methods in a completely different approach.

For $n>1$, the problem of characterizing pairs of projections $P_{\phi_1}, P_{\phi_2}$ on $H^2(\D^n)$ for which the commutator $[P_{\phi_1},P_{\phi_2}]$ has finite--rank is difficult. As it will become apparent from our approach, the first major obstacle is the absence of an analogue of Coburn's theorem (for instance, see \cite{CIL}). Secondly, the lack of a fair understanding of when two inner functions on $\D^n$ are \text{co-prime}. This is another place where we make a contribution by identifying a new sufficient criterion for inner functions to be co-prime (see Definition \ref{co-pr}). Surprisingly, in certain cases, we can bypass Coburn's theorem by using this co-prime property of inner functions. In Theorem \ref{polyfinite}, we prove that if $\phi_1$ and $\phi_2$ are co-prime, then $[P_{\phi_1}, P_{\phi_2}]$ has finite--rank if and only if it is identically zero. Equivalently, this occurs if and only if $ \phi_1,\phi_2$ depend on disjoint sets of variables. A particular case where this is applicable is when $\phi_1, \phi_2$ are inner functions in the bidisc algebra $\mathcal{A}(\D^2)$. Specifically, in Corollary \ref{disc_finite}, we prove that if $\phi_1, \phi_2 \in \mathcal{A}(\D^2)$, then $[P_{\phi_1}, P_{\phi_2}]$ has finite--rank if and only if $[P_{\phi_1}, P_{\phi_2}]=0$. 

Let us now outline the organization of this article. In Section \ref{sec1}, we review Halmos' two projections theorem for abstract projections and establish few preliminary results concerning the projections $P_{\phi_1}$ and $P_{\phi_2}$ on $H^2(\D)$ that will be used throughout the paper. Section \ref{sec2} is devoted to the characterization of finite-rank commutators. As an application of our main result, we refine Bessonov's criterion \cite{RV} for the characterization of finite-rank truncated Toeplitz operators in the case of inner symbols (see Theorem \ref{tto}). In Section \ref{sec3}, we characterize when a pair of projections $(P_{\phi_1},P_{\phi_2})$ on $H^2(\D)$ with finite-rank commutator becomes a Fredholm pair. Next, in section \ref{sec5}, we find a sufficient criterion for inner functions to be co-prime over $\D^n$ and establish explicit cases for co-primeness. Lastly, in Section \ref{sec6}, we prove that for co-prime inner functions on $\D^n$, the commutator $[P_{\phi_1}, P_{\phi_2}]$ has finite rank if and only if it is trivial. Moreover, we characterize when the product of projections $(I_{H^2(\D^n)} - P_{\phi_1}) (I_{H^2(\D^n)} - P_{\phi_2})$ become a compact operator on $H^2(\D^n)$.

\section{Preliminaries}\label{sec1}

In this section, our aim is to describe the commutator $[P_{\phi_1}, P_{\phi_2}]$ corresponding to inner functions $\phi_1, \phi_2$ using Halmos’s two projections theorem. First, let us briefly explain Halmos's characterization for abstract Hilbert spaces. Let $P_1,P_2$ be two orthogonal projections onto closed subspaces $\clm, \cln \subseteq \mathcal{H}$, respectively. The first step is to find an orthogonal decomposition of $\clh$ using componenets of these subspaces. In particular, let us denote,
\[
\cll_{00} :=\clm \cap \cln,\quad \cll_{01} := \clm \cap \cln^{\perp},\quad \cll_{10} :=\clm^{\perp}\cap \cln,\quad \cll_{11} :=\clm^{\perp}\cap \cln^{\perp}.
\]
Moreover, if $\cll_0:={\clm}\ominus (\cll_{00}\oplus \cll_{01})$ and $\cll_1 :={\clm}^{\perp}\ominus (\cll_{10}\oplus \cll_{11})$, Halmos observed that
\begin{equation}\label{Halmos_decomp}
 \mathcal{H}=\cll_{00}\oplus \cll_{01}\oplus \cll_{10}\oplus \cll_{11}\oplus \cll_0\oplus \cll_1.
\end{equation}
Using this decomposition, Halmos \cite[Theorem 2]{Halmos} found representations of the projections $P_1, P_2$ in the following manner (see \cite[Theorem 1.2]{BS} for this description).
\begin{thm}[Halmos \cite{Halmos}]\label{Halmos_two_subspace}
If one of the spaces $\cll_0$ and $\cll_1$ is non-trivial, then $\mbox{dim } \cll_0 = \mbox{dim } \cll_1$. Moreover, there exists a unitary operator $U:\cll_1\rightarrow \cll_0$ such that
   
    \begin{align*}
        {P_1}&=(I_{\cll_{00}},I_{\cll_{01}},0_{\cll_{10}},0_{\cll_{11}})\oplus {W}^*\begin{bmatrix}
            I & 0 \\
            0 & 0
        \end{bmatrix}{W},\notag\\
        {P_2}&=(I_{\cll_{00}},0_{\cll_{01}},I_{\cll_{10}},0_{\cll_{11}})\oplus {W}^*\begin{bmatrix}
            {T} & {\sqrt{T(I -T)}}\\
            {\sqrt{T(I - T)}} & {I -T}
        \end{bmatrix}{W},
    \end{align*}
where ${W}=\text{diag}(I_{\cll_0},{U})$.  Here ${T}:\cll_0\rightarrow \cll_0$ is the positive contraction defined by $$T=P_{{\cll_0}}P_2|_{\cll_0},$$ satisfying $\ker~T=\ker~(I-{T})=\{0\}.$
\end{thm}

This following few corollaries (well-known) are simple application of Halmos' theorem.
\begin{cor}\label{K(I-K)}
Let $(P_1, P_2)$ be a pair of orthogonal projections on $\clh$. Then the commutator $[P_1,P_2]$ is of finite--rank if and only if $T(I - T)$ is of finite--rank.
\end{cor}
    \begin{proof}
        Using Theorem \ref{Halmos_two_subspace}, we get
        \begin{equation}\label{commutator}
        P_1P_2-P_2P_1=(0,0,0,0)\oplus W^*\begin{bmatrix}
            0 & \sqrt{T(I-T)}\\
            -\sqrt{T(I-T)}& 0
        \end{bmatrix}W.
 \end{equation}
    Since $T(I-T)$ is a positive operator, $[P_1, P_2]$ is a finite--rank operator if and only if $T(I-T)$ is a finite--rank operator.
    \end{proof}
    
      \begin{cor}\label{K_finite}
Let $(P_1, P_2)$ be a pair of orthogonal projections on $\clh$, then \\ \vspace{1mm}
$(i)$~the commutator $[P_{\phi_1}, P_{\phi_2}]$ is finite-rank if and only if $\cll_0$ is finite-dimensional;\\ \vspace{1mm}
$(ii)$~ $[P_{\phi_1}, P_{\phi_2}]=0$ if and only if $\cll_0=\{0\}.$
\end{cor}
    \begin{proof}
Let us start with the proof for $(i)$. By the above corollary, we already have that $[P_1,P_2]$ has finite--rank if and only if $T(I - T)$ has finite-rank.  Moreover, since $T$ and $I- T$ both are injective operators. Thus, $T(I-T)$ has finite-rank if and only if the domain of $T(I-T)$, that is, $\cll_0$ is finite dimensional. The proof of $(ii)$ follows from \cite[Proposition 1.5]{BS}.
\end{proof}
    
    \begin{cor}\label{K finite}
Let $(P_1, P_2)$ be a pair of orthogonal projections on $\clh$. Then the product $(I - P_1)(I - P_2)$ is compact (or finite--rank) if and only if $(I - P_1) \clh \cap (I - P_2)\clh$ is finite-dimensional and $T$ is compact (or finite--rank).
\end{cor}
\begin{proof}
Using Theorem \ref{Halmos_decomp}, we get
       \begin{equation*} 
 	\begin{split}
       (I-P_1)(I-P_2)
        &=I-P_1-P_2+P_1P_2\\
        &=(0_{\cll_{00}},0_{\cll_{01}},0_{\cll_{10}},I_{\cll_{11}})\oplus \begin{bmatrix}
            0 & 0\\
            -\sqrt{T(I_{}-T)} & T
        \end{bmatrix}.
        \end{split}
       \end{equation*}
Thus, we can conclude that if $(I - P_1)(I - P_2)$ is compact (or finite--rank), $(I - P_1) \clh \cap (I - P_2)\clh$ is finite-dimensional and $\begin{bmatrix}
            0 & 0\\
            -\sqrt{T(I-T)} & T
        \end{bmatrix}$ is compact (or finite--rank). Moreover,
         \[
        \begin{bmatrix}
            0 & 0\\
            -\sqrt{T(I-T)} & T
        \end{bmatrix}\begin{bmatrix}
            0 & 0\\
            -\sqrt{T(I-T)} & T
        \end{bmatrix}^*=\begin{bmatrix}
            0 & 0\\
            0 & T
        \end{bmatrix},
\] 
shows that $\begin{bmatrix}
            0 & 0\\
            -\sqrt{T(I-T)} & T
        \end{bmatrix}$ is compact (or finite-rank) if and only if $T$ is compact (or finite-rank). This completes the proof.
\end{proof}

We will now focus on understanding Halmos's decomposition for shift-invariant subspaces of the Hardy space $H^2(\D)$. In particular, let us consider a pair of $T_z$-invariant closed subspaces $\cls_1, \cls_2 \subseteq H^2(\D)$.  By Beurling's theorem \cite{Beurling} there exists inner functions $\phi_1, \phi_2$ such that $\cls_i = \phi_i H^2(\D)$ for $i=1,2$. Thus, $
P_{\cls_1}  = T_{\phi_1} T_{\phi_1}^*$ and $P_{\cls_2} = T_{\phi_2} T_{\phi_2}^*.$
For the sake of computation, let us denote these projections as $$P_{\phi_1}:= P_{\cls_1} \text{ and } P_{\phi_2}:= P_{\cls_2}. $$

It is evident from Halmos's decomposition that there is a choice on whether we consider $P_{\phi_1}$  as the first projection and $P_{\phi_2}$ as the second, or vice versa. Since we are concerned with the finite-rank commutator $[P_{\phi_1}, P_{\phi_2}]$, the particular choice does not affect whether the commutator has finite--rank.  But for the sake of simplifying the decomposition (\ref{Halmos_decomp}), as will be evident later, we choose a convention. In particular, it is well-known from Coburn's theorem \cite{Coburn} that either
\[
\ker T_{\bar \phi_2 \phi_1} = \{0\} \quad \text{or }\quad \ker T_{\bar \phi_1 \phi_2} = \{0\}. 
\]
\begin{equation}\label{convention}
\textsf{Convention: } P_1 = P_{\phi_1}, P_2 = P_{\phi_2} \textsf{ when } 
\ker T_{\bar \phi_2 \phi_1} = \{0\}.
\end{equation}
So, without the loss of any generality let us assume $P_1 = P_{\phi_1}$ and $P_2 = P_{\phi_2}$, and proceed with Halmos's decomposition. In particular, 
\[
\clm=\text{ range } (P_{\phi_1})=\phi_1 H^2(\mathbb{D}) \text{ and } \cln=\text{ range }(P_{\phi_2})=\phi_2 H^2(\mathbb{D}).
\]
 Consequently, following \cite[Corollary 4.8]{GMR}, we obtain
\begin{equation*} 
    \cll_{00} =~\phi_1 H^2(\mathbb{D} )\cap\phi_2 H^2(\mathbb{D} )=\mbox{lcm}(\phi_1,\phi_2)H^2({\mathbb{D}} )
\end{equation*}
and following \cite[Corollary 5.9]{GMR}, we get
\begin{equation*}
 \cll_{11}=~\left(\phi_1 H^2(\mathbb{D} )\right)^{\perp}\cap\left(\phi_2 H^2(\mathbb{D} )\right)^{\perp}=\left(\mbox{gcd}(\phi_1, \phi_2)H^2({\mathbb{D}} )\right)^{\perp},
\end{equation*}
where the \textit{greatest common divisor} of a pair of inner functions $\phi_1, \phi_2$ or $\mbox{gcd }(\phi_1, \phi_2)$ is defined to be the greatest inner function $\phi$ that divides both $\phi_1$ and $\phi_2$.  
Moreover,
\begin{align*}
\left(\phi_1 H^2(\mathbb{D} )\right)\cap\left(\phi_2 H^2(\mathbb{D} )\right)^{\perp}
    =&~\{\phi_1f:f\in H^2(\mathbb{D} )~\text{and}~\phi_1f\perp \phi_2h~\text{for all }h\in H^2(\mathbb{D} )\}\\
    =&~\{\phi_1f:f\in H^2(\mathbb{D} )~\text{and}~T_{\phi_2}^*T_{\phi_1}f\perp h~\text{for all }h\in H^2(\mathbb{D} )\}\\
    =&~\{\phi_1f:f\in \ker~T_{\bar \phi_2 \phi_1}\},
\end{align*}
and hence,
\begin{equation}\label{L_01}
\cll_{01}= \phi_1 H^2(\mathbb{D}) \cap (\phi_2 H^2(\mathbb{D} ))^{\perp} = \phi_1 \ker~T_{\bar \phi_2 \phi_1}.
\end{equation}
Thus, from our convention based on the Coburn's theorem we get $\cll_{01} = \{0\}$, and therefore,
\begin{equation}\label{L_0}
\cll_0 =~\clm \ominus \big( \cll_{00} \oplus \cll_{01} \big)
    =~\phi_1H^2(\D)\ominus \big(\mbox{lcm}(\phi_1,\phi_2)H^2({\mathbb{D}} ) \big).
\end{equation}
Next, if we assume that $\mbox{gcd}(\phi_1, \phi_2) = 1$,  it is equivalent to $\mbox{lcm}(\phi_1, \phi_2)= \phi_1 \phi_2 $, and thus,
\begin{equation}\label{new L_0}
\cll_0 
    =~\phi_1H^2(\D) \ominus \phi_1 \phi_2H^2(\D) =~\phi_1 \big(H^2(\D) \ominus \phi_2H^2(\D)\big) = \phi_1 \ker T_{\bar \phi_2},
\end{equation}
The following simple observation shows that the assumption $\mbox{gcd}(\phi_1, \phi_2)=1$, has no effect on the finite-rank property of the commutator $[P_{\phi_1}, P_{\phi_2}]$.
\begin{lem}\label{gcd}
Let $\phi_1,\phi_2$ be two inner functions on $\D$, and $\phi=\mbox{gcd }(\phi_1,\phi_2)$. Consider $\psi_i = \frac{\phi_i}{\phi}$, for $i=1,2$. Then $[P_{\phi_1}, P_{\phi_2}]$ has finite--rank if and only if $[P_{\psi_1}, P_{\psi_2}]$ has finite-rank and $$\text{rank of }[P_{\psi_1}, P_{\psi_2}]=\text{rank of }[P_{\phi_1}, P_{\phi_2}].$$
\end{lem}
\begin{proof}
Since $\phi_i = \phi \psi_i$ for $i=1,2$, we get
\[
T_{\phi_1}^* T_{\phi_2}  = T_{\psi_1}^*  T_{\phi}^*  T_{\phi} T_{\psi_2} = T_{\psi_1}^* T_{\psi_2}.
\]
Thus, we have
\begin{align*}
[P_{\phi_1}, P_{\phi_2}]  &= T_{\phi_1} T_{\phi_1}^* T_{\phi_2} T_{\phi_2}^* - T_{\phi_2} T_{\phi_2}^*T_{\phi_1} T_{\phi_1}^*\\
&= T_{\phi_1} T_{\psi_1}^* T_{\psi_2} T_{\phi_2}^* - T_{\phi_2} T_{\psi_2}^* T_{\psi_1} T_{\phi_1}^*\\
&= T_{\phi} \left( T_{\psi_1} T_{\psi_1}^* T_{\psi_2} T_{\psi_2}^* - T_{\psi_2} T_{\psi_2}^* T_{\psi_1} T_{\psi_1}^* \right) T_{\phi}^*\\
&= T_{\phi}[P_{\psi_1}, P_{\psi_2}]  T_{\phi}^*.
\end{align*}
Since, $\phi$ is an inner function, $T_{\phi}$ is an isometry, and therefore, $[P_{\phi_1}, P_{\phi_2}] $ is of finite-rank if and only if $[P_{\psi_1}, P_{\psi_2}]$ is of finite-rank. Finally, let us first observe that the above equality implies
\[
\phi \ker [P_{\psi_1}, P_{\psi_2}] \subseteq \ker [P_{\phi_1}, P_{\phi_2}] \text{ and } \phi \mbox{ ran } [P_{\psi_1}, P_{\psi_2}] \subseteq \mbox{ran } [P_{\phi_1}, P_{\phi_2}].
\]
These inclusions further imply that $\mbox{rank }[P_{\phi_1}, P_{\phi_2}] \leq \mbox{rank }[P_{\psi_1}, P_{\psi_2}]$ and $\mbox{rank }[P_{\psi_1}, P_{\psi_2}] \leq \mbox{rank }[P_{\phi_1}, P_{\pi_2}]$, which together implies the equality of the ranks. This completes the proof.
\end{proof}

\section{ finite-rank commutator}\label{sec2}

In this section, we provide a complete characterization of the finite-rank commutator $[P_{\phi_1}, P_{\phi_2}]$  associated with the inner functions $\phi_1, \phi_2$. Since most of the results will involve finite Blaschke products, let us recall that an inner function $\phi$ on $\D$ is a finite Blaschke product if there exists a $\zeta \in \mathbb{T}$, $n \in \mathbb{N}$ and $\{a_1, \ldots, a_n\} \in \D$ such that
\[
\phi(z) = \zeta \prod_{i=1}^n \frac{z - a_i}{1 - \bar a_i z} \quad (z \in \D).
\]
We begin by establishing the following two results, which will serve as essential tools in the subsequent analysis.

\begin{lem}\label{zeroes}
Let $\phi_1, \phi_2$ be distinct inner functions on $\D$.  If $\phi_2$ is a finite Blaschke product and $\mbox{deg } \phi_2 \leq \mbox{deg }\phi_1$, then $\ker T_{\bar \phi_2 \phi_1} = 0$.
\end{lem}
\begin{proof}
Suppose $h\in \ker T_{\bar\phi_2\phi_1}$, then $\phi_1 h \in \ker T_{\phi_2}^*$. Let $\mbox{deg }\phi_2 = n$ and $\{\alpha_1,\alpha_2,\ldots,\alpha_n\}$ be the zeros of  $\phi_2$, counted with multiplicity. Then following \cite[Corollary 5.18]{GMR},  we get     
\begin{align*} 
                 \ker T_{\phi_2}^* = H^2(\mathbb{D})\ominus \phi_2 H^2(\mathbb{D})=\left\{\frac{a_0+a_1z+a_2z^2+\cdots+a_{n-1}z^{n-1}}{(1-\bar\alpha_1z)(1-\bar\alpha_2z)\cdots(1-\bar\alpha_nz)}:a_j\in\mathbb{C}\right\}.
\end{align*} 
Thus, any function in $\ker T_{\phi_2}^*$ must have at most $n-1$ zeroes and therefore, $\phi_1 h$ has at most $n-1$ zeroes. Now if $\mbox{deg}(\phi_1)=\infty$, then obviously $\phi_1h\notin \ker T^*_{\phi_2}.$ On the other hand, if $\mbox{deg}(\phi_1)$ is finite then $\phi_1 h$ has at least $n$-many zeroes since $\mbox{deg }\phi_1 \geq \mbox{deg }\phi_2=n$. This is a contradiction to $\phi_1 h \in \ker T_{\phi_2}^*$ unless $h=0$, which implies that $\ker T_{\bar \phi_2 \phi_1} = 0$.  This completes the proof.
\end{proof}
 {\begin{lem}\label{phi1,phi2 finite}
Let $\phi_1,\phi_2$ be two inner functions on $\D$ with $\mbox{gcd }(\phi_1,\phi_2)=1$ and $\ker T_{\bar\phi_2\phi_1}=\{0\}$. If $\phi_2$ is a finite Blaschke product then $\phi_2H^2\cap (H^2\ominus \phi_1 H^2)$ is a finite-dimensional subspace if and only if $\phi_1$ is a finite Blaschke product. 
\end{lem}
\begin{proof}
If $\phi_1$ is a finite Blaschke then obviously $\phi_2H^2\cap (H^2\ominus \phi_1 H^2)$ is a finite dimensional space. For proving the other direction, suppose $\phi_1$ is not a finite Blasckhe product but $\phi_2H^2\cap (H^2\ominus \phi_1 H^2)$ is finite dimensional. Since $\ker T_{\bar\phi_2\phi_1}=\{0\}$, following our convention (\ref{convention}) and (\ref{Halmos_decomp}), we get
\[
\clm^{\perp} = (\phi_1 H^2(\mathbb{D}))^\perp=\cll_{11} \oplus \cll_{10} \oplus \cll_1.
\]
Here, $\cll_{11}=(\phi_1H^2(\mathbb{D}))^\perp\cap(\phi_2H^2(\mathbb{D}))^\perp = (\mbox{gcd }(\phi_1, \phi_2) H^2(\mathbb{D}))^\perp  = \{0\}$. Also, $\phi_1H^2(\mathbb{D})^\perp$ is an infinite-dimensional subspace since $\phi_1$ is assumed not to be a finite Blaschke product. Moreover, since $\phi_2$ is finite Blaschke, then $\cll_0=\phi_1(H^2(\mathbb{D}\ominus \phi_2H^2(\mathbb{D})))$ is a finite-dimensional space (by \ref{new L_0}). Hence, by Theorem \ref{Halmos_two_subspace}, $\cll_1$ is also finite-dimensional space. This further implies that $\cll_{10}=\phi_2H^2\cap (H^2\ominus \phi_1 H^2)$ must be finite-dimensional, which is a contradiction to our assumption. This completes the proof.
\end{proof}}
We are now ready to prove the main result on finite-rank commutator. For this purpose, let us recall that a Hankel operator $H_\phi: H^2(\D) \raro H^2(\D)$ is defined by $H_{\phi} := P_{H^2(\D)^{\perp}} L_{\phi}|_{H^2(\D)}$ corresponding to $\phi \in L^{\infty}$, where $L_{\phi}$ is the multiplication operator on $L^2(\mathbb{T})$.
\begin{thm}\label{finite-rank}
Let $\phi_1, \phi_2$ be distinct inner functions on $\D$ with ${\mbox{gcd }(\phi_1, \phi_2) = 1}$. Then the following are equivalent:\\ \vspace{1mm}
$(i)$~the commutator $[P_{\phi_1}, P_{\phi_2}]$ has finite-rank;\\ \vspace{1mm}
$(ii)$~$\phi_1$ or $\phi_2$ is a finite Blaschke product;\\
$(iii)$~$H_{\bar \phi_1}$ or $H_{\bar \phi_2}$ is finite-rank.\\ \vspace{1mm} \vspace{1mm}
Moreover, in this case, 
\[
\text{rank of } [P_{\phi_1}, P_{\phi_2}] =2 \text{min}\{\deg (\phi_1), \deg (\phi_2)\}.
\] 
\end{thm}
\begin{proof}
 From Corollary \ref{K_finite}, it is clear that if $P_1 = P_{\phi_1}$ and $P_2 = P_{\phi_2}$, then $[P_1, P_2]$ is finite-rank if and only if $\cll_0 = \phi_1 \ker T_{\phi_2}^*$ is finite-dimensional, which is again equivlalent to $\ker T_{\phi_2}^*$ being finite-dimensional and therefore, $\phi_2$ must be a finite Blaschke product. Similarly, if $P_1 = P_{\phi_2}$ and $P_2 = P_{\phi_1}$, then $[P_1, P_2]$ is finite-rank if and only if $\phi_1$ is a finite Blaschke product. This proves $(i)$ implies $(ii)$. For proving the other direction, suppose $\phi_1$ is a finite Blaschke product, then either $\mbox{deg }(\phi_1) \leq  \mbox{deg }(\phi_2)$ or $\mbox{deg }(\phi_2) \leq  \mbox{deg }(\phi_1)$. Using Lemma \ref{zeroes}, we can deduce that in the first case $\ker T_{\bar \phi_1 \phi_2}=0$ and in the second case, $\phi_2$ must also be a finite Blaschke product and hence, $\ker T_{\bar \phi_2 \phi_1}=0$. Depending on these cases, based on our convention (\ref{convention}), we can choose $P_1 = P_{\phi_2}$ and $P_2 = P_{\phi_1}$ or $P_1 = P_{\phi_1}$ and $P_2 = P_{\phi_2}$, respectively. Thus, using  (\ref{new L_0}), we get
$
\cll_0=T_{\phi_2}(I_{H^2(\D)}-P_{\phi_1})H^2(\D)
$ or $
\cll_0=T_{\phi_1}(I_{H^2(\D)}-P_{\phi_2})H^2(\D)
$, respectively. In both the cases, we get $\cll_0$ is finite-dimensional, and hence, by Corollary \ref{K_finite} again, $[P_{\phi_1}, P_{\phi_2}]$ must have finite-rank. This proves the equivalence between $(i)$ and $(ii)$. The equivalence between $(ii)$ and $(iii)$ follows from the observation that $\ker H_{\bar \theta} = \theta H^2(\D)$ for any inner function $\theta$ on $\D$.  For proving, the identity on $\mbox{rank }[P_{\phi_1}, P_{\phi_2}]$, let us assume without the loss of any generality that $\phi_2$ is a finite Blaschke product and $\text{deg}(\phi_2)=n\leq \text{deg}(\phi_1)$. Following, Lemma \ref{zeroes}, we get $\ker T_{\bar\phi_2\phi_1} =\{0\}$, and therefore, using identity  (\ref{L_01}), we get $\cll_{01} = \{0\}$. Thus, using  (\ref{new L_0}), we get
\[
\cll_0=T_{\phi_1}(I_{H^2(\D)}-P_{\phi_2})H^2(\D),
\]
and since $\phi_2$ is a Blaschke product with $\mbox{deg }(\phi_2)=n$, therefore dim $\cll_0=n$. Thus, $T=P_{\cll_0}P_{\phi_2}|_{\cll_0}$ is a finite-rank operator. Moreover by Theorem \ref{Halmos_two_subspace}, $\ker (T)=\ker (I-T)=\{0\}$, which implies that $\mbox{rank }T(I - T) = \dim \cll_0 = \dim \ker T_{\phi_2}^* = n$.  Now, from the expression of the commutator (\ref{commutator}) in Corollary \ref{K(I-K)}, we get 
\[
\text{rank of } [P_{\phi_1}, P_{\phi_2}] = 2 \mbox{ rank }\sqrt{T(I-  T)} = 2n.
\]
Similarly, if $\phi_1$ is a finite Blaschke product and $\mbox{deg } \phi_1 = m \leq \mbox{deg } \phi_2 $, we shall get
\[
\text{rank of } [P_{\phi_1}, P_{\phi_2}] =  2m.
\]
Therefore, $\mbox{rank } [P_{\phi_1}, P_{\phi_2}] =  2 \min\{ \mbox{deg } (\phi_1), \mbox{deg } (\phi_2)\}$. 
\end{proof}
The following result is an immediate consequence of the previous theorem and Lemma \ref{gcd}, and addresses the case in which $\phi_1\text{ and }\phi_2$ are not co-prime.
 \begin{cor}\label{without gcd1}
Let $\phi_1, \phi_2$ be distinct inner functions on $\D$ with $\mbox{gcd }(\phi_1, \phi_2) = \phi$. Furthermore, let $\psi_i = \tfrac{\phi_i}{\phi}$, for $i=1,2$. Then the following are equivalent:\\ \vspace{1mm}
$(i)$~the commutator $[P_{\phi_1}, P_{\phi_2}]$ has  finite-rank,\\ \vspace{1mm}
$(ii)$~$\psi_1$ or $\psi_2$ is a finite Blaschke product,\\
$(iii)$~$H_{\bar \psi_1}$ or $H_{\bar \psi_2}$ is finite-rank.\\ \vspace{1mm} \vspace{1mm}
Moreover, in this case, 
\[
\text{rank of } [P_{\phi_1}, P_{\phi_2}] = 2 \text{min}\{\deg (\psi_1), \deg (\psi_2)\}.\]
 \end{cor}
\begin{cor}
Let $\phi_1, \phi_2$ be distinct co-prime inner functions on $\D$ such that $[P_{\phi_1}, P_{\phi_2}]$ is finite-rank, then $$\mbox{rank }[P_{\phi_1}, P_{\phi_2}] = 2 \mbox{ rank } H_{\bar \phi_1} H_{\bar \phi_2}.$$
\end{cor}
\begin{proof}
Following \cite[Theorem 2]{Richman}, a formula for the rank of the product of Hankel operators $H_{\phi}$ and $H_{\psi}$ corresponding to $\phi, \psi \in L^{\infty}(\T)$ is given by
\[
\mbox{rank } H_{\phi} H_{\psi} = \min \{\mbox{rank } H_{\phi}, \mbox{rank } H_{\psi}\}.
\]
Furthermore, when $\phi_1, \phi_2$ are inner functions, then $\mbox{deg }\phi_i = \mbox{rank }H_{\bar \phi_i}$, for $i=1,2$ and thus,
\[
\mbox{rank } H_{\bar \phi_1} H_{\bar \phi_2} = \min \{\mbox{deg }\phi_1, \mbox{deg }\phi_2\}.
\]
By using Theorem \ref{finite-rank},  we get $$\mbox{rank }[P_{\phi_1}, P_{\phi_2}] = 2 \mbox{ rank } H_{\bar \phi_1} H_{\bar \phi_2},$$
when $[P_{\phi_1}, P_{\phi_2}]$ is finite-rank. This completes the proof.
\end{proof}
The following corollary shows that \cite[Theorem 2.2]{DPS}, which addresses the case of commuting pair of projections $(P_{\phi_1},P_{\phi_2})$, is an immediate consequence of Theorem \ref{finite-rank}.
\begin{cor}
Let $\phi_1, \phi_2$ be distinct inner functions on $\D$ such that $\phi=\text{ gcd }(\phi_1,\phi_2)$. Then $[P_{\phi_1},P_{\phi_2}]=0$ if and only if $\frac{\phi_1}{\phi}$ or $\frac{\phi_2}{\phi}$ is constant.
\end{cor}
\begin{proof}
Let $\psi_1=\frac{\phi_1}{\phi}$ and $\psi_2=\frac{\phi_2}{\phi}$, then from $$[P_{\phi_1},P_{\phi_2}]=T_{\phi}[P_{\psi_1}, P_{\psi_2}]  T_{\phi}^*$$ it follows that $[P_{\phi_1},P_{\phi_2}]=0$ if and only if $[P_{\psi_1},P_{\psi_2}]=0$. Now if $[P_{\psi_1},P_{\psi_2}]=0$, then rank of $[P_{\psi_1},P_{\psi_2}]=0$. Therefore by Theorem \ref{finite-rank}, min $\{\text{deg}(\psi_1),\text{ deg}(\psi_2)\}=0$, and hence one of $\psi_1,\psi_2$ is constant. This completes the proof.
\end{proof}
We conclude this section by applying Theorem \ref{finite-rank} to characterize finite-rank truncated Toeplitz operators with inner symbols. The general problem of classifying finite-rank truncated Toeplitz operators was solved by Bessonov \cite{RV}. Our result can be perceived as a refinement in the case of inner symbols. Let us recall that for any inner function $\theta$ on $\D$, let $\clk_{\theta}:= H^2 \ominus \theta H^2$. Then for any $\phi \in L^2(\T)$, the truncated Toeplitz operator $T_{\phi}^{\theta}$ on $\clk_{\theta}$ is defined by $T_{\phi}^{\theta}f = P_{\clk_{\theta}} \phi f$ for $f \in \clk_{\theta}$.
\begin{thm}\label{tto}
Let $\phi_1,\phi_2$ be distinct inner functions with $\phi = \mbox{gcd}(\phi_1, \phi_2)$. Then the truncated Toeplitz operator $T_{\phi_1}^{\phi_2}$ on $\clk_{\phi_2}$ has finite--rank if and only if $\tfrac{\phi_2}{\phi}$ is a finite Blaschke product.
\end{thm}
\begin{proof}
Let $\psi_i=\frac{\phi_i}{\phi}$ for $i=1,2$. Suppose that $\psi_2$ is a finite Blaschke product. Then $T_{\psi_1}^{\psi_2}$ is clearly of finite--rank, since
$(T_{\psi_1}^{\psi_2})^*=T_{\psi_1}^*(I-P_{\psi_2})$ has finite--rank. Consequently, $(T_{\phi_1}^{\phi_2})^*=T_{\psi_1}^*(I-P_{\psi_2})T_{\phi}^*$ implies that $(T_{\phi_1}^{\phi_2})^*$ has finite--rank and thus, $T_{\phi_1}^{\phi_2}$ has finite--rank. Conversely, suppose that $T^{\phi_2}_{\phi_1}$ is of finite--rank. Since $(T_{\psi_1}^{\psi_2})^*=(T^{\phi_2}_{\phi_1})^*T_{\phi}$ it follows that $(T_{\psi_1}^{\psi_2})^*$ is also has finite--rank. Moreover, from the identity 
\begin{align*}
T_{\psi_1}(T_{\psi_1}^{\psi_2})^*-\left(T_{\psi_1}(T_{\psi_1}^{\psi_2})^*\right)^*&=T_{\psi_1}T^*_{\psi_1}-P_{\psi_1}P_{\psi_2}-T_{\psi_1}T^*_{\psi_1}+P_{\psi_2}P_{\psi_1}\\
&=P_{\psi_2}P_{\psi_1}-P_{\psi_1}P_{\psi_2},
\end{align*}
it follows that $[P_{\psi_1},P_{\psi_2}]$ also has finite--rank. Hence by Theorem \ref{finite-rank} at least one of $\psi_1,\psi_2$ is a finite Blaschke product. If possible, let $\psi_1$ be a finite Blaschke product and $\psi_2$ not be a finite Blaschke product. Then by Lemma \ref{zeroes}, $\ker T_{\bar\psi_1\psi_2}=\{0\}$. Therefore using Lemma \ref{phi1,phi2 finite}, we get $\psi_1 H^2(\D)\cap (\psi_2 H^2(\D))^\perp$ is an infinte dimensional space. Now, following our Convention \ref{convention} and Theorem \ref{Halmos_two_subspace}  we have
\begin{equation*}
T_{\psi_1}(T_{\psi_1}^{\psi_2})^*=P_{\psi_1}(I-P_{\psi_2})=(0_{\cll_{00}},0_{\cll_{01}},I_{\cll_{10}},0_{\cll_{11}})\oplus {W}^*\begin{bmatrix}
            0 & {\sqrt{T(I - T)}}\\
             0 & {I -T}
        \end{bmatrix}{W},
\end{equation*}
where $\cll_{10}=\psi_1 H^2(\D)\cap (\psi_2 H^2(\D))^\perp$. This contradicts the assumption that $T_{\psi_1}^{\psi_2}$ has finite--rank. This completes the proof.
\end{proof}

\section{Fredholm pairs}\label{sec3}
In this section, we investigate when does a  pair of projections $(P_{\phi_1}, P_{\phi_2})$ with finite--rank commutator become a Fredholm pair. The following result \cite[Proposition 3.1]{ASS} gives an useful characterization.
\begin{lem}\label{B.Simon}
    $(P,Q)$ is a Fredholm pair if and only if 
    \begin{enumerate}
        \item $1$ and $-1$ are isolated points of $\sigma(P-Q)$,
        \item $\ker (P-Q\pm I)$ are finite dimensional.
    \end{enumerate}
    Moreover, \begin{align*}
        \text{index }(P,Q)=\text{dim}\ker (P-Q-I)-\text{dim}\ker (P-Q+I).
    \end{align*}
\end{lem}

\begin{prop}[\cite{ASS}]\label{Fredholm}
A necessary and sufficient condition that $(P, Q)$ be a Fredholm pair is that $P - Q = F+D$, where $F, D$ are self-adjoint, $\|D\| < 1$ and $F$ is finite-rank.
\end{prop}
Let us now note that the Fredholmness is invariant under the division by common factors.
\begin{lem}\label{Fredholm_pair}
Let $\phi_1,\phi_2$ be distinct   inner functions on $\D$. Moreover, let $\phi=\mbox{gcd }(\phi_1,\phi_2)$ and $\psi_i = \frac{\phi_i}{\phi}$, for $i=1,2$. Then $(P_{\phi_1}, P_{\phi_2})$ is  a Fredholm pair if and only if $(P_{\psi_1}, P_{\psi_2})$ is a Fredholm pair.
\end{lem}
\begin{proof}
Let us first note that
\[
P_{\phi_1} - P_{\phi_2} = T_{\phi_1}T_{\phi_1}^* - T_{\phi_2}T_{\phi_2}^* = T_{\phi}T_{\psi_1} T_{\psi_1}^* T_{\phi}^* - T_{\phi}T_{\psi_2}T_{\psi_2}^*T_{\phi}^* = T_{\phi} \big( P_{\psi_1} - P_{\psi_2}  \big) T_{\phi}^*.
\]
If $(P_{\psi_1}, P_{\psi_2})$ is a Fredholm pair then by Proposition \ref{Fredholm}, there exists self-adjoint operators $F, D$ satisfying
\[
P_{\psi_1} - P_{\psi_2} = F + D.
\]
Thus,
\[
P_{\phi_1} - P_{\phi_2} =  T_{\phi} FT_{\phi}^* + T_{\phi}DT_{\phi}^*.
\]
Now $F,D$ being self-adjoint imply that $T_{\phi} FT_{\phi}^*$ and $T_{\phi}DT_{\phi}^*$ are self-adjoint operators.  Moreover, $T_{\phi} FT_{\phi}^*$ is finite-rank since $F$ is finite-rank and $\|T_{\phi}DT_{\phi}^* \| < 1$ since $\|D\| < 1$. Thus, again by Proposition \ref{Fredholm}, $(P_{\phi_1}, P_{\phi_2})$ must be a Fredholm pair. Conversely, if $(P_{\phi_1}, P_{\phi_2})$ is a Fredholm pair, then there exists self-adjoint operators $\tilde{F}, \tilde{D}$, where $\tilde{F}$ is finite-rank, $\|\tilde{D}\| <1$ and
\[
P_{\phi_1} - P_{\phi_2} = \tilde{F} + \tilde{D},
\]
and therefore,
\[
P_{\psi_1} - P_{\psi_2} = T_{\phi}^* \big( P_{\phi_1} - P_{\phi_2} \big)T_{\phi} = T_{\phi}^* \tilde{F}T_{\phi}+ T_{\phi}^* \tilde{D}T_{\phi}.
\]
Thus, again by Proposition \ref{Fredholm}, $(P_{\psi_1}, P_{\psi_2})$ becomes a Fredholm pair. This completes the proof.
\end{proof}


\begin{lem}\label{deg}
     Let $\phi_1,\phi_2$ be distinct and finite Blaschke products. Then, $\mbox{dim }(\ker T_{\bar\phi_1\phi_2})=\text{max}~\{0,\text{deg}(\phi_1)-\text{deg}(\phi_2)\}$.
     \end{lem}
     \begin{proof}
Let $\phi=\mbox{gcd}(\phi_1,\phi_2)$, and set $\psi_i=\tfrac{\phi_i}{\phi}$ for $i=1,2$, then obviously $\ker T_{\bar\psi_1\psi_2}=\ker T_{\bar\phi_1\phi_2}$. Hence, without loss of generality, we may assume that $\mbox{gcd}(\phi_1,\phi_2)=1$.
Suppose $\alpha_1,\alpha_2,\ldots,\alpha_m$ and $\beta_1,\beta_2,\ldots,\beta_n$ are the zeros of $\phi_1$ and $\phi_2$, respectively, counted according to their multiplicities. Then it is well known that dim $(H^2\ominus\phi_1H^2)=m$. Moreover, from (\ref{L_01}), we already know that
\[
\text{dim }(\ker T_{\bar\phi_1\phi_2})=
              \text{dim }((\phi_1 H^2(\D))^\perp \cap \phi_2H^2(\D)).
\]
         Now following \cite[ Corollary 5.18]{GMR}, we have 
             \begin{equation}\label{finite_Blaschke}
                 H^2(\mathbb{D})\ominus \phi_1 H^2(\mathbb{D})=\left\{\frac{a_0+a_1z+a_2z^2+\cdots+a_{m-1}z^{m-1}}{(1-\bar\alpha_1z)(1-\bar\alpha_2z)\cdots(1-\bar\alpha_mz)}:a_j\in\mathbb{C}\right\}.
             \end{equation}
            Suppose $f\in \phi_1 H^2(\D)^\perp\cap \phi_2H^2(\D)$. Then $f(z)=\frac{p(z)}{(1-\bar\alpha_1z)(1-\bar\alpha_2z)\cdots(1-\bar\alpha_mz)}$, where $\mbox{deg}~(p(z))\leq m- 1$. Moreover $f\in \phi_2H^2(\D)$, implies that \begin{align}\label{deg of c}
                p(z)=\prod_{i=1}^{n}(z-\beta_i)c(z),
            \end{align} 
            for some polynomial $c$. Thus if $n>m$ then there does not exist any $c$ satisfying \ref{deg of c} and if $n\leq m$ then deg $c(z)\leq m-n-1$.
            Now for $n\leq m$

            \begin{align*}
                \{\frac{\left(\prod_{i=1}^{n}(z-\beta_i)\right)z^k}{(1-\bar\alpha_1z)(1-\bar\alpha_2z)\cdots(1-\bar\alpha_mz)}:k=0,1,\ldots m-n-1\}\subset (\phi_1 H^2(\D))^\perp\cap(\phi_2 H^2(\D)),
            \end{align*}
since these functions belong in $H^2(\D) \ominus \phi_1 H^2(\D)$ by the above description in (\ref{finite_Blaschke}) and also,
\[
\frac{\left(\prod_{i=1}^{n}(z-\beta_i)\right)z^k}{(1-\bar\alpha_1z)(1-\bar\alpha_2z)\cdots(1-\bar\alpha_mz)} = \prod_{i=1}^n \frac{z-\beta_i}{1 - \bar \beta_i z} \, \boldsymbol{\cdot} \, \frac{\prod_{i=1}^{n}{1 - \bar \beta_i z}}{(1-\bar\alpha_1z)(1-\bar\alpha_2z)\cdots(1-\bar\alpha_mz)}.
\]
Moreover for any $k \in \mathbb{N}$,
 \begin{align*}
                &\frac{\left(\prod_{i=1}^{n}(z-\beta_i)\right)z^{k}}{(1-\bar\alpha_1z)(1-\bar\alpha_2z)\cdots(1-\bar\alpha_mz)}\notin \\
                &\text{span}\{\frac{\left(\prod_{i=1}^{n}(z-\beta_i)\right)z^j}{(1-\bar\alpha_1z)(1-\bar\alpha_2z)\cdots(1-\bar\alpha_mz)}:j=0,1,\ldots,k-1\}.
            \end{align*}
            
       This implies that $\{\frac{\left(\prod_{i=1}^{n}(z-\beta_i)\right)z^k}{(1-\bar\alpha_1z)(1-\bar\alpha_2z)\cdots(1-\bar\alpha_mz)}:k=0,1,\ldots m-n-1\}$    forms a basis for   $(\phi_1 H^2)^\perp\cap(\phi_2H^2)$.
            Hence dim $(\phi_1 H^2)^\perp\cap(\phi_2H^2)=\text{max }\{0,m-n\}$. 
     
     \end{proof}
    
We will now prove that pairs of the projections  $(P_{\phi_1}, P_{\phi_2})$ with finite--rank commutator is a Fredholm if and only if both $\phi_1$ and $\phi_2$ are finite Blaschke products.  

\begin{thm}\label{mainfredhlom}
Let $\phi_1,\phi_2$ be  co-prime inner functions on $\D$ such that $[P_{\phi_1},P_{\phi_2}]$ is a finite--rank operator. Then $(P_{\phi_1},P_{\phi_2})$ is a Fredhlom pair if and only if both $\phi_1$ and $\phi_2$ are finite Blaschke products.  Moreover, 
\begin{align*}
        \text{index }(P_{\phi_1},P_{\phi_2})=\text{deg}(\phi_2)-\text{deg}(\phi_1).
    \end{align*}
    \end{thm}
    \begin{proof}
Assume $\ker T_{\bar\phi_2\phi_1}=\{0\}$. By definition $(P_{\phi_1},P_{\phi_2})$ is a Fredhlom pair if $P_{\phi_2}P_{\phi_1}$ is a Fredhlom operator from range of $P_{\phi_1}$ to range of $P_{\phi_2}$. Now from the representation of $P_{\phi_1},P_{\phi_2}$ in Theorem \ref{Halmos_two_subspace} we have 
\begin{align}\label{QP}
{P_{\phi_2}P_{\phi_1}}&=(I_{\cll_{00}},0_{\cll_{01}},0_{\cll_{10}},0_{\cll_{11}})\oplus {W}^*\begin{bmatrix}
           T & 0 \\
            \sqrt{T(I-T)} & 0
        \end{bmatrix}{W}. 
\end{align}  
\begin{align}\label{QP2}
{P_{\phi_1}P_{\phi_2}}&=(I_{\cll_{00}},0_{\cll_{01}},0_{\cll_{10}},0_{\cll_{11}})\oplus {W}^*\begin{bmatrix}
           T & \sqrt{T(I-T)} \\
            0 & 0
        \end{bmatrix}{W}. 
\end{align}  
Here $\cll_{11}=\{0\}$ since gcd$(\phi_1,\phi_2)=1$.  Moreover $[P_{\phi_1},P_{\phi_2}]$ has finite rank implies that $T$ is a finite rank operator, therefore from (\ref{QP}) it is clear that range of $P_{\phi_2}P_{\phi_1}|_{P_{\phi_1}H^2}$ is closed. Thus, $(P_{\phi_1},P_{\phi_2})$ is a Fredhlom pair if and only if $\ker (P_{\phi_2}P_{\phi_1}|_{P_{\phi_1}H^2})$  and $\ker (P_{\phi_1}P_{\phi_2}|_{P_{\phi_2}H^2})$ both are finite dimensional. Therefore, from (\ref{QP}) and (\ref{QP2}), it follows that $(P_{\phi_1},P_{\phi_2})$ is a Fredhlom pair if and only if $\cll_{01}$ and $\cll_{10}$ both are finite dimensional. Finally by  Lemma \ref{phi1,phi2 finite}, we conclude that $\cll_{01}$ and $\cll_{10}$ both are finite dimensional if and only if both $\phi_1,\phi_2$ are finite Blaschke product.

Let us now compute the index. First, let us note that
$$\begin{bmatrix}
            x\\
            y
        \end{bmatrix}\in \ker(\begin{bmatrix}
            {-T} & {-\sqrt{T(I-T)}}\\
            {-\sqrt{T(I-T)}} & {T-2I}
        \end{bmatrix})$$ if and only if $$\begin{bmatrix}
            -y\\
            x
        \end{bmatrix}\in 
        \ker(\begin{bmatrix}
            {2-T} & {-\sqrt{T(I-T)}}\\
            {-\sqrt{T(I-T)}} & {T}
        \end{bmatrix}).$$ Thus, we have
        \begin{align*}
            \text{dim }&(\ker(\begin{bmatrix}
            {-T} & {-\sqrt{T(I-T)}}\\
            {-\sqrt{T(I-T)}} & {T-2I}
        \end{bmatrix}))\\
        &~=\text{dim }(\ker(\begin{bmatrix}
            {2-T} & {-\sqrt{T(I-T)}}\\
            {-\sqrt{T(I-T)}} & {T}
        \end{bmatrix})).
        \end{align*}
        For proving the identity on the index of the pair $(P_{\phi_1}, P_{\phi_2})$, using Theorem \ref{Halmos_two_subspace}, let us note that
        \begin{equation}\label{eqns1}
     \begin{split}
            {P_{\phi_1}-P_{\phi_2}-I}&=(-I_{\cll_{00}},0,-2I_{\cll_{10}},-I_{\cll_{11}})\\
            &~\oplus {W}^*\begin{bmatrix}
            {-T} & {-\sqrt{T(I-T)}}\\
            {-\sqrt{T(I-T)}} & {T-2I}
        \end{bmatrix}{W}.
        \end{split}
        \end{equation}
        
        \begin{equation}\label{eqns2}
        \begin{split}
        {P_{\phi_1}-P_{\phi_2}+I}&=(I_{\cll_{00}},2I_{\cll_{01}},0,I_{\cll_{11}})\\&~\oplus {W}^*\begin{bmatrix}
            {2-T} & {-\sqrt{T(I-T)}}\\
            {-\sqrt{T(I-T)}} & {T}
        \end{bmatrix}{W}.
        \end{split}
        \end{equation}
        Thus, $\cll_{10} \subseteq \ker ( P_{\phi_1} - P_{\phi_2} + I)$ and $\cll_{01} \subseteq \ker ( P_{\phi_1} - P_{\phi_2} - I)$, and hence,
\begin{align*}
            \text{dim }&\ker({P_{\phi_1}-P_{\phi_2}-I})-\text{ dim }\ker({P_{\phi_1}-P_{\phi_2}+I})\\
            =&\text{ dim }(\cll_{01})-\text{ dim }(\cll_{10})\\
             =&\text{ dim }(\phi_1 \ker T_{\bar \phi_2 \phi_1} )- \text{ dim }(\phi_2 \ker T_{\bar \phi_1 \phi_2})\\
            =&\text{ max}\{0,\text{deg}(\phi_2)-\text{deg}(\phi_1)\}-\text{max}\{0,\text{deg}(\phi_1)-\text{deg}(\phi_2)\}\\
            =&~\text{deg}(\phi_2)-\text{deg}(\phi_1),
\end{align*}
where the third equality follows from Lemma \ref{deg}. Thus the final result follows from Lemma \ref{B.Simon}.  This completes the proof.
\end{proof}
\begin{cor}\label{lastfredhlom}
Let $\phi_1,\phi_2$ be inner functions on $\D$ such that $[P_{\phi_1},P_{\phi_2}]$ is a finite--rank operator. Then $(P_{\phi_1},P_{\phi_2})$ is a Fredhlom pair if and only if both $\frac{\phi_1}{\phi}$ and $\frac{\phi_2}{\phi}$ are finite Blaschke products, where $\phi=\mbox{gcd}(\phi_1,\phi_2)$.  Moreover 
\begin{align*}
        \text{index }(P_{\phi_1},P_{\phi_2})=\mbox{deg}(\tfrac{\phi_2}{\phi})-\text{deg}(\tfrac{\phi_1}{\phi}).
    \end{align*}
\end{cor}
\begin{proof}
From equations (\ref{eqns1}) and (\ref{eqns2}), it is clear that even if we do not assume $\mbox{gcd }(\phi_1, \phi_2)=1$, the subspace $\cll_{11}$ does not contribute anything to $\ker ( P_{\phi_1} - P_{\phi_2} + I)$ or $\ker ( P_{\phi_1} - P_{\phi_2} - I)$. Therefore, 
\[
\mbox{index }(P_{\phi_1}, P_{\phi_2}) = \mbox{index }(P_{\frac{\phi_1}{\phi}}, P_{\frac{\phi_2}{\phi}}),
\]
where $\phi = \mbox{gcd }(\phi_1, \phi_2)$. The proof now follows Theorem \ref{mainfredhlom} and Lemma \ref{Fredholm_pair}.
\end{proof}
 
\section{Co-prime inner functions on {$\D^n$}}\label{sec5}
It is well-known that inner functions on $\D^n$ do not admit a well-behaved gcd/lcm theory. Our main goal in this section is to study certain properties of co-primeness that may still be valid on $\D^n$. Let us first recall that $H^2(\D^n)$ is the space of all analytic functions $f(\bm{z})$ on the unit polydisc $\D^n$ such that 
\[
\|f\|: = \sup_{0 \leq r<1} \big( \int_{\mathbb{T}^n} |f(r \bm{z})|^2 d\mu \big)^{\frac{1}{2}} < \infty,
\]
where $\bm{z}:= (z_1,\ldots,z_n) \in \D^n$, and $\mu$ is the normalized Lebesgue measure on $\mathbb{T}^n$, the distinguished boundary of $\D^n$. The algebra of bounded analytic functions on $\D^n$ is denoted by $H^{\infty}(\D^n)$. A function $\phi \in H^{\infty}(\D^n)$ is inner if $|\phi(\bm{\mu})| = 1 $ for almost everywhere $\bm{\mu} \in \mathbb{T}^n$. In the case of $H^2(\D)$, it is well-known that for inner functions, $\phi_1,\phi_2$ on $\D$ the following are equivalent \cite{GMR}.\\ \vspace{1mm}
$(i)$~$\mbox{gcd }(\phi_1,\phi_2)=1$; \\ \vspace{1mm}
$(ii)$~$\phi_1H^2\cap \phi_2 H^2=\phi_1\phi_2 H^2$;\\ \vspace{1mm}
$(iii)$~$\phi_1H^2\bigvee \phi_2 H^2=H^2$; { (``$\bigvee$" denotes the span closure)}\\ \vspace{1mm}
$(iv)$~$\phi_1(H^2\ominus\phi_2H^2)\cap\phi_2H^2=\{0\}$.

There are several examples (like the following) which show that these conditions do not remain equivalent on $H^2(\D^n)$ when $n>1$.
\begin{ex}
Consider two inner functions $\phi_1=z_1,~\phi_2=z_2\in H^2(\mathbb{D}^2).$ Since every function  $f\in \phi_1H^2\bigvee \phi_2H^2$ vanishes at $(0,0)$, we have $f(0,0)=0$. Therefore, $\phi_1H^2\bigvee \phi_2H^2\neq H^2$. However, $z_1(H^2\ominus z_2H^2)\cap z_2H^2=\{0\}$.
\end{ex}
We will show that $(ii)$ and $(iv)$ remains equivalent even on $\D^n$. 

\begin{prop}\label{lcm}
Let $\phi_1, \phi_2$ be distinct inner functions on $\D^n$. Then 
\begin{equation*} 
(\phi_1 H^2(\D^n) \cap \phi_2 H^2(\D^n)) \ominus \phi_1 \phi_2 H^2(\D^n) 
=  \phi_{1} \ker T_{\phi_2}^* \cap \phi_2 \ker T_{\phi_1}^* 
\end{equation*}
\end{prop}
\begin{proof}
Let us begin by observing that
\[
(\phi_1 H^2(\D^n) \cap \phi_2 H^2(\D^n)) \ominus \phi_1 \phi_2 H^2(\D^n) = \phi_1 H^2(\D^n) \cap \phi_2 H^2(\D^n) \cap \ker T_{\phi_1}^* T_{\phi_2}^*.
\]
Furthermore, if $f \in \phi_1 H^2(\mathbb{D} )\cap\phi_2 H^2(\mathbb{D} ) \cap \ker T_{\phi_1}^* T_{\phi_2}^* $, then $
(I_{H^2(\D^n)} - T_{\phi_1} T_{\phi_2} T_{\phi_1}^* T_{\phi_2}^*) f = f$.
But the following expressions,
\[
(I_{H^2(\D^n)} - T_{\phi_1} T_{\phi_2} T_{\phi_1}^* T_{\phi_2}^*) = I_{H^2(\D^n)} - T_{\phi_1}T_{\phi_1}^*+ T_{\phi_1}(I_{H^2(\D^n)}  - T_{\phi_2}T_{\phi_2}^*)T_{\phi_1}^*.
 \]
 and 
 \[
(I_{H^2(\D^n)} - T_{\phi_1} T_{\phi_2} T_{\phi_1}^* T_{\phi_2}^*) = I_{H^2(\D^n)} - T_{\phi_2}T_{\phi_2}^*+ T_{\phi_2}(I_{H^2(\D^n)}  - T_{\phi_1}T_{\phi_1}^*)T_{\phi_2}^*.
 \]
imply that
\[
T_{\phi_1}(I_{H^2(\D^n)}  - T_{\phi_2}T_{\phi_2}^*)T_{\phi_1}^* f = f = T_{\phi_2}(I_{H^2(\D^n)}  - T_{\phi_1}T_{\phi_1}^*)T_{\phi_2}^* f.
\]
Thus,
\[
\phi_1 H^2(\mathbb{D} )\cap\phi_2 H^2(\mathbb{D} ) \cap \ker T_{\phi_1}^* T_{\phi_2}^* \subseteq \mbox{ran } T_{\phi_1} (I_{H^2(\D^n)} - T_{\phi_2}T_{\phi_2}^*)  \cap \mbox{ran } T_{\phi_2}(I_{H^2(\D^n)} - T_{\phi_1}T_{\phi_1}^*).
\]
Both $T_{\phi_1} (I_{H^2(\D^n)} - T_{\phi_2}T_{\phi_2}^*)$ and $T_{\phi_2} (I_{H^2(\D^n)} - T_{\phi_1}T_{\phi_1}^*)$  are partial isometries and moreover,
$
\mbox{ran } T_{\phi_1} (I_{H^2(\D^n)} - T_{\phi_2}T_{\phi_2}^*) = \mbox{ran } T_{\phi_1} (I_{H^2(\D^n)} - T_{\phi_2}T_{\phi_2}^*) T_{\phi_1}^* = \phi_1 \ker T_{\phi_2}^*$ and $ \mbox{ran } T_{\phi_2} (I_{H^2(\D^n)} - T_{\phi_1}T_{\phi_1}^*) = \mbox{ran } T_{\phi_2} (I_{H^2(\D^n)} - T_{\phi_1}T_{\phi_1}^*) T_{\phi_2}^* = \phi_2 \ker T_{\phi_1}^*$. Conversely, if $f \in \phi_1 \ker T_{\phi_2}^* \cap \phi_2 \ker T_{\phi_1}^*$, then it is clear that $f \in  \phi_1 H^2(\D^n) \cap \phi_2 H^2(\D^n) \cap \ker T_{\phi_1}^* T_{\phi_2}^*.$ This completes the proof.
\end{proof}

\begin{thm}
    Let $\phi_1,\phi_2$ be distinct inner functions on $\mathbb{D}^n$, then the following are equivalent.\\ \vspace{2mm}
        $(i)$~$\phi_1 H^2(\mathbb{D}^n)\cap \phi_2 H^2(\mathbb{D}^n)=\phi_1\phi_2H^2(\mathbb{D}^n)$,\\ \vspace{1mm}
          $(ii)$~$\phi_1 \ker T_{\phi_2}^* \cap \phi_2 \ker T_{\phi_1}^*=\{0\},$\\ \vspace{1mm}
          $(iii)$~$ \phi_1 \ker T_{\phi_2}^* \cap \phi_2 H^2(\mathbb{D}^n)=\{0\}.$
        \end{thm}
    \begin{proof}
From Proposition \ref{lcm}, it is clear that $(i) \Longleftrightarrow (ii)$. For proving, the equivalence between $(i)$ and $(iii)$, let us first assume $\phi_1 H^2(\mathbb{D}^n)\cap \phi_2 H^2(\mathbb{D}^n)=\phi_1\phi_2H^2(\mathbb{D}^n)$ and consider 
$
   f \in \phi_1 \ker T_{\phi_2}^* \cap \phi_2 H^2(\mathbb{D}^n).
$
 Then $f=\phi_1 h_1$ for some $h_1\in H^2(\mathbb{D}^n)\ominus \phi_2H^2(\mathbb{D}^n)$, also the assumption further implies that $f\in \phi_1 H^2(\mathbb{D}^n)\cap \phi_2 H^2(\mathbb{D}^n)=\phi_1\phi_2H^2(\mathbb{D}^n).$ Therefore $f=\phi_1\phi_2g$, for some $g\in H^2(\mathbb{D}^n). $ Hence, we have $\phi_1h_1=\phi_1\phi_2g$, which implies $h_1=\phi_2g$. But this is a contradiction to the fact that $h_1\in H^2(\mathbb{D}^n)\ominus \phi_2H^2(\mathbb{D}^n)$, unless $h_1=0$, which would further imply $f=0$. Thus, $$\phi_1(H^2(\mathbb{D}^n)\ominus \phi_2H^2(\mathbb{D}^n))\cap \phi_2H^2(\mathbb{D}^n)=\{0\}.$$

For the other direction, we just need to show that $\phi_1 \ker T_{\phi_2}^* \cap \phi_2 H^2(\mathbb{D}^n)=\{0\}$ implies $ \phi_1 H^2(\mathbb{D}^n)\cap\phi_2 H^2(\mathbb{D}^n) \subseteq \phi_1\phi_2 H^2(\mathbb{D}^n)$ since the other inclusion always holds. Suppose $f\in \phi_1H^2(\D^n)\cap\phi_2H^2(\mathbb{D}^n)$, then $$f=\phi_1h_1\text{ for some }h_1\in H^2(\mathbb{D}^n).$$ Moreover,  $h_1=\phi_2g+k,$ for some $g \in H^2(\D^n)$ and $k\in (\phi_2H^2)^\perp$ using the orthogonal decomposition $H^2(\D^n) = \phi_2 H^2(\D^n) \oplus \ker T_{\phi_2}^*$. This further implies that $$f=\phi_1h_1=\phi_1\phi_2g+\phi_1k.$$ Hence, we get $$\phi_1k=f-\phi_1\phi_2g,$$ which implies $\phi_1k\in \phi_2H^2(\D^n).$ But $k\in (\phi_2H^2)^\perp$, implies that $$\phi_1k\in \phi_1(H^2(\mathbb{D}^n)\ominus \phi_2H^2(\mathbb{D}^n))\cap \phi_2H^2(\mathbb{D}^n)=\{0\}.$$ Thus, $\phi_1k=0$, which further implies $k=0$ and hence, $f=\phi_1h_1=\phi_1\phi_2g$.  This completes the proof.
    \end{proof}

Motivated by the above characterization, we frame the following definition.
\begin{defn}\label{co-pr}
For $n>1$, a pair of inner functions $\phi_1, \phi_2$ on $\D^n$ is said to be \textit{co-prime} if $\phi_1 H^2(\D^n) \cap \phi_2 H^2(\D^n) =\phi_1 \phi_2 H^2(\D^n)$.
\end{defn}
Let us now identify few cases, where a pair of inner functions becomes co-prime.
\begin{thm}\label{prod_inner}
Any pair of inner functions $\phi_1,\phi_2$ on $\mathbb{D}^n$ satisfying $\ker T_{\phi_1}^* \cap \ker T_{\phi_2}^* = \{0\}$ is co-prime.
    \end{thm}
    \begin{proof}
The inclusion $\phi_1\phi_2H^2(\D^n) \subseteq \phi_1H^2(\D^n) \cap\phi_2 H^2(\D^n)$ is always satisfied. For the other direction, suppose $f\in \phi_1H^2(\D^n) \cap\phi_2H^2(\D^n)$. Now, $\ker T_{\phi_1}^* \cap \ker T_{\phi_2}^* = \{0\}$ implies that $\phi_1 H^2(\D^n) \bigvee \phi_2 H^2(\D^n)=H^2(\D^n)$ . Therefore, there exist $p_k,q_k\in H^2(\D^n)$ such that $$\phi_1p_k+\phi_2q_k\rightarrow 1,\text{in }H^2(\D^n).$$ Using Cauchy-Schwarz we get $$f\phi_1p_k+f\phi_2q_k\rightarrow f\text{ in }H^1(\D^n).$$ Since $f\in \phi_1H^2(\D^n) \cap \phi_2H^2(\D^n)$, then $f=\phi_1h_1=\phi_2h_2$, for some $h_1,h_2\in H^2(\D^n)$. Using this we get $$f\phi_1p_k+f\phi_2q_k=\phi_1\phi_2(h_2p_k+h_1q_k)\rightarrow f.$$ Moreover, $\phi_1\phi_2H^1(\D^n)$ is a closed subspace in $H^1(\D^n)$, thus $f\in \phi_1\phi_2H^1(\D^n)$. Let $f=\phi_1\phi_2 g$, for some $g\in H^1(\D^n)$. Since $|\phi_i|=1$ a.e. $\mathbb{T}^n$, then it implies $|f|=|g|$ a.e. $\mathbb{T}^n$, and since $f\in H^2(\D^n)$ therefore it follows that $g\in H^2(\D^n)$. Hence we have $f=\phi_1\phi_2g$ for some $g\in H^2(\D^n)$. Thus $\phi_1 H^2(\D^n) \cap \phi_2 H^2(\D^n) \subset \phi_1\phi_2H^2(\D^n)$. This completes the proof.
    \end{proof}

\begin{lem}
Let $\phi_1$ and $\phi_2$ be distinct inner functions on $\mathbb{D}^n$. If $\phi_1$ and $\phi_2$ depend on disjoint sets of variables, then $\phi_1 \ker T_{\phi_2}^* \cap \phi_2 \ker  T_{\phi_1}^*=\{0\}$.
\end{lem}    
\begin{proof}
Suppose $f\in \phi_1 \ker T_{\phi_2}^* \cap \phi_2 \ker T_{\phi_1}^*. $ Then $f=T_{\phi_1}h_1=T_{\phi_2}h_2,$ where $h_1\in\ker T_{\phi_2}^*$ and $h_2\in \ker T_{\phi_1}^*$. Now here $\phi_1,\phi_2$ depend on disjoint sets of variables then by \cite[Corollary 2]{Gu}, $T^*_{\phi_2}T_{\phi_1}=T_{\phi_1}T^*_{\phi_2}$. Moreover $h_1\in \ker T^*_{\phi_2}$ then we have 
\begin{align*}
h_2&=T^*_{\phi_2}T_{\phi_2}h_2\\&=T^*_{\phi_2}T_{\phi_1}h_1\\&=T_{\phi_1}T^*_{\phi_2}h_1=0.
\end{align*}
 Hence $f=T_{\phi_2}h_2=0,$ and $\phi_1 \ker T_{\phi_2}^* \cap \phi_2 \ker T^*_{\phi_1}=\{0\}$.
\end{proof} 
Next, we consider another family of co-prime inner functions introduced by Zhang et al. in \cite{ZTLYZ}. For the sake of completeness, we present an alternative proof of the corresponding result.
\begin{thm}\label{co-primes}
Let $\phi_1, \phi_2$ be inner functions on $\D^n$, where $n>1$. If $\phi_1 H^2(\D^n) + \phi_2 H^2(\D^n)$ has finite co-dimension, then $\phi_1H^2\cap\phi_2H^2=\phi_1\phi_2H^2.$
\end{thm}
\begin{proof}
The $n \geq 3$ case, is a straightforward application of Ahern and Clark's result \cite[Corollary]{AC}, which states that the invariant subspace generated by a finite collection of functions $\{f_1, \ldots, f_k\} \in H^2(\D^n)$, where $k<n$, is either the full space or has infinite co-dimension. Thus, $\phi_1 H^2(\D^n) + \phi_2 H^2(\D^n)$ can have finite co-dimension only when $\phi_1 H^2(\D^n) + \phi_2 H^2(\D^n) = H^2(\D^n)$, which is equivalent to saying that $\ker T_{\phi_1}^* \cap \ker T_{\phi_2}^* = \{0\}$ Then, the conclusion follows by using Theorem \ref{prod_inner}. 

The $n=2$ case follows verbatim as in the proof of \cite[Theorem 3.5]{GW}. The only thing to note is that the assumption $\phi_1 H^2(\D^2) + \phi_2 H^2(\D^2)$ has finite co-dimension, implies that $\ker T_{\phi_1}^* \cap \ker T_{\phi_2}^*$ is finite-dimensional. This further implies that 
\[
Z(\phi_1) \cap Z(\phi_2) \cap \D^2,
\]
is a finite set, where $Z(\phi)$, corresponding to any analytic function $\phi$ on $\D^2$, denotes the set of all zeroes of $\phi$ inside $\D^2$. This completes the proof.
\end{proof}
For the next result, we need to recall some facts about rational inner function from the influential book by Walter Rudin \cite{Rudin}. The following definition are from the excellent recent survey by Knese \cite{Knese}. 
\begin{defn}
Given polynomials $p,q \in \mathbb{C}[z_1, \ldots, z_n]$ with no common factors, where $q(\bm{z}) \neq 0$ for $\bm{z} \in \D^n$, the function $\phi = p/q$ is a rational inner function if $|\phi(\bm{\mu})| = 1 $ for almost every $\bm{\mu} \in \mathbb{T}^n$.
\end{defn}
Rudin \cite[Theorem 5.2.5]{Rudin} observed that any rational inner function $\phi$ can be re-written as the following (also see \cite[Theorem 1]{Knese})
\[
\phi(\bm{z}) = \bm{z}^{\bm{k}} \frac{\tilde{p}(\bm{z})}{p(\bm{z})}, \quad (\bm{z} \in \D^n, \bm{k} \in \mathbb{N}^n),
\]
where $\bm{z}^{\bm{k}} = z_1^{k_1} \cdots z_n^{k_n}$, and $\tilde{p}$ is the reflection of the polynomial $p$ defined by $\tilde{p}(\bm{z}) = \bm{z}^{\bm{m}}\overline{p(\frac{1}{\bar z_1}, \ldots, \frac{1}{\bar z_n})}$. Here $\bm{m} = \mbox{deg}(p) = (\mbox{deg}_{z_1}(p), \ldots, \mbox{deg}_{z_n}(p))$ is the multi-degree of $p$. 
Let us quickly observe a simple (probably well-known) fact useful for the sequel.
\begin{lem}\label{common}
Let $\phi_1 = \bm{z}^{\bm{k}}\frac{\tilde{p}(\bm{z})}{p(\bm{z})}$ and $\phi_2 =  \bm{z}^{\bm{l}} \frac{\tilde{q}(\bm{z})}{q(\bm{z})}$ be rational inner functions on $\D^n$. If $\phi_1$ and $\phi_2$ does not have any common inner factor, then $\tilde{p}, \tilde{q}$ does not have any common factor.
\end{lem}
\begin{proof}
It is clear that $\bm{k}$ and $\bm{l}$ when realized as a subset of $\{1,\ldots,n\}$ must be disjoint, if the rational functions do not have any common inner factor. Furthermore, if we assume $r$ is a non-constant polynomial which is a common factor for $p$ and $q$. Then $\phi_1$ and $\phi_2$ will have the common inner factor $\tilde{r}/r$. This will be a contradiction to the assumption that $\phi_1$ and $\phi_2$ do not have any common inner factor. Now, it is well-known that: $p, q$ does not have a common factor if and only if the reflections $\tilde{p}, \tilde{q}$ does not have a common factor. It is a straightforward consequence of the following properties (see \cite[Section 2]{AJMJES}): $(a)$~$\widetilde{p_1 p_2}=\tilde{p_1}\tilde{p_2}$, for polynomials $p_1,p_2$, and $(b)$~$\tilde{\tilde{p}}=p$, since $\tilde{p}$ is reflection of $p$ at its multi-degree.This completes the proof.
\end{proof}
An interesting sub-collection of rational inner functions comes from the polydisc algebra. Recall that the polydisc algebra $\mathcal{A}(\D^n)$ is the collection of holomorphic functions on $\D^n$ that are continuous on $\overline{\D^n}$. Rudin in \cite[Theorem 5.2.5]{Rudin}, remarkably observed that any inner function $\phi \in \mathcal{A}(\D^n)$ is actually a rational inner function $\phi(z) =  \bm{z}^{\bm{k}}\frac{\tilde{p}(\bm{z})}{p(\bm{z})}$ such that $p(\bm{z})$ has no zero on $\overline{\D^n}$. The advantage in this case, is that $1/p \in H^{\infty}(\D^n)$. Let us now present a few results essential for the sequel.
\begin{thm}\label{disc_case}
Let $\phi_1, \phi_2$ be inner functions inside $\mathcal{A}(\D^2)$. If $\phi_1$ and $\phi_2$ do not have any common inner factor. Then $\phi_1$ and $\phi_2$ are co-prime.
\end{thm}
\begin{proof}
It is clear from the assumption that $\phi_1=\bm{z}^{\bm{k}} \frac{\tilde{p}}{p}$ and $\phi_2 = \bm{z}^{\bm{l}}\frac{\tilde{q}}{q}$ for some polynomials $p,q$ which do not have any zeroes on $\overline{\D^2}$. It is clear from Lemma \ref{common} that if $\phi_1$ and $\phi_2$ do not have any common inner factor then $\tilde{p}$ and $\tilde{q}$ do not have any common factor. Moreover, since, we are working on $\D^2$, the assumption will force that either $\bm{k} = (m,0)$ and $\bm{l} = (0,n)$, or vice-versa for $m,n \in \mathbb{N}$. In particular, we get $\bm{z}^{\bm{k}} \tilde{p}$ and $\bm{z}^{\bm{l}} \tilde{q}$ do not have any common factor. Therefore, following \cite[ Lemma 3.7, Step 2]{GW}, we get that $\bm{z}^{\bm{k}}\tilde{p} H^2(\D^2) + \bm{z}^{\bm{l}}\tilde{q} H^2(\D^2)$ is a subspace of finite co-dimension and hence, it is also closed. Now, $p, q$ have no zero on $\overline{\D^2}$ imply that $1/p, 1/q \in H^{\infty}(\D^2)$. Thus,
\[
\phi_1 H^2(\D^2) + \phi_2 H^2(\D^2) =  \bm{z}^{\bm{k}}\frac{\tilde{p}}{p} H^2(\D^2) + \bm{z}^{\bm{l}}\frac{\tilde{q}}{q} H^2(\D^2) \subseteq \bm{z}^{\bm{k}}\tilde{p} H^2(\D^2) + \bm{z}^{\bm{l}}\tilde{q} H^2(\D^2),
\]
and hence, $\phi_1 H^2(\D^2) + \phi_2 H^2(\D^2) = \bm{z}^{\bm{k}}\tilde{p} H^2(\D^2) + \bm{z}^{\bm{l}}\tilde{q} H^2(\D^2)$. Thus, $\phi_1 H^2(\D^2) + \phi_2 H^2(\D^2)$ has finite co-dimension. Therefore, by Theorem \ref{co-primes}, we get $\phi_1$ and $\phi_2$ are co-prime. This completes the proof.
\end{proof}

\section{On Hardy space over the unit polydisc}\label{sec6}
 We begin this section by addressing the main question of determining when the commutator $[P_{\phi_1}, P_{\phi_2}]$ has finite--rank. Recall that using the Halmos decomposition (recall the expression in (\ref{L_0})) for the pair of projections $(P_{\phi_1}, P_{\phi_2})$ we got
\[
\cll_0 =~\clm \ominus \big( \cll_{00} \oplus \cll_{01} \big),
\]
where
\[
\cll_{00} = \phi_1 H^2(\D^n) \cap \phi_2 H^2(\D^n) \text{ and } \cll_{01}= \left(\phi_1 H^2(\mathbb{D}^n )\right)\cap\left(\phi_2 H^2(\mathbb{D}^n )\right)^{\perp} = \phi_1 \ker~T_{\bar \phi_2 \phi_1}.
\]
So, if we take a pair of \textit{co-prime} inner functions $\phi_1, \phi_2$, then we get
\[
\cll_0 = \phi_1 H^2(\D^n) \ominus \big(\phi_1 \phi_2 H^2(\D^n) \oplus \phi_1 \ker T_{\bar \phi_2 \phi_1}\big) = \phi_1 \big(\ker T_{\phi_2}^* \ominus \ker T_{\bar \phi_2 \phi_1} \big)
\]
Let us first collect an useful result.
\begin{lem}\label{lem_com}
Let $\phi_1, \phi_2$ be distinct inner functions on $\D^n$. Then $$ \ker T_{\bar \phi_2} \ominus \ker~T_{\bar \phi_2 \phi_1} = \overline{ran  } [T_{\bar \phi_1}, T_{\phi_2}].$$
\end{lem}
\begin{proof}
Let us begin by noting that $\ker T_{\bar \phi_2 \phi_1} = \ker T_{\phi_2}^* T_{\phi_1} \subseteq \ker T_{\phi_2}^*$. This is true because, if $f \in \ker T_{\phi_2}^* T_{\phi_1} $, then $T_{\phi_2}^* T_{\phi_1}  f = 0$, therefore, multiplying on the left hand side by $T_{\phi_1}^*$ gives $T_{\phi_2}^* f = 0$. Now, to prove the equality in the statement, we will first prove that $\mbox{ran } [T_{\bar \phi_1}, T_{\phi_2}] \subseteq \ker T_{\bar \phi_2} \ominus \ker~T_{\bar \phi_2 \phi_1}$.  Let $f \in \mbox{ran } [T_{\bar \phi_1}, T_{\phi_2}]$, so there exists a $g \in H^2(\D^n)$ such that
\[
f = (T_{\bar \phi_1 \phi_2} - T_{\phi_2} T_{\bar \phi_1})g.
\]
So, $T_{\phi_2}T_{\phi_2}^* f = T_{\phi_2}T_{\phi_2}^* (T_{\phi_1}^* T_{\phi_2} - T_{\phi_2} T_{\bar \phi_1})g = (T_{\phi_2} T_{\phi_1}^*  - T_{\phi_2} T_{\bar \phi_1})g = 0$. Thus, $f \in \ker T_{\phi_2}^*$. Furthermore, let $h \in \ker T_{\bar \phi_2 \phi_1}  = \ker T_{\phi_2}^* T_{\phi_1} \subseteq \ker T_{\phi_2}^*$. Then
\[
\langle f , h \rangle_{H^2(\D^n)} = \langle (T_{\bar \phi_1 \phi_2} - T_{\phi_2} T_{\bar \phi_1})g, h \rangle_{H^2(\D^n)} = \langle (g, T_{\bar \phi_2 \phi_1} - T_{\phi_1} T_{\bar \phi_2}) h \rangle_{H^2(\D^n)} =0,
\]
proves that $f \in \ker T_{\phi_2}^* \ominus \ker T_{\bar \phi_2 \phi_1}$. Thus, we have proved $\mbox{ran } [T_{\bar \phi_1}, T_{\phi_2}] \subseteq \ker T_{\bar \phi_2} \ominus \ker~T_{\bar \phi_2 \phi_1}$ and therefore, $\overline{\mbox{ran }} [T_{\bar \phi_1}, T_{\phi_2}] \subseteq \ker T_{\bar \phi_2} \ominus \ker~T_{\bar \phi_2 \phi_1}$. For the opposite inclusion, we will prove that $\ker T_{\phi_2}^* \ominus \ker T_{\bar \phi_2 \phi_1} = \overline{\mbox{ran}} \, T_{\bar \phi_1 \phi_2} \cap \ker T_{\phi_2}^* \subseteq \overline{\mbox{ran}} [T_{\bar \phi_1}, T_{\phi_2}]$. Let $x \in \overline{\mbox{ran}} \, T_{\bar \phi_1 \phi_2} \cap \ker T_{\phi_2}^*$, so there must exist $\{y_n\} \in H^2(\D^n)$ such that $x = \lim_{n \rightarrow \infty}T_{\phi_1}^* T_{\phi_2} y_n$. Since $x \in \ker T_{\phi_2}^*$, we get
\[
\lim_{n \rightarrow \infty} T_{\phi_1}^* y_n = \lim_{n \rightarrow \infty} T_{\phi_2}^* T_{\phi_1}^* T_{\phi_2} y_n =  T_{\phi_2}^* x  = 0.
\]
This implies $x = \lim_{n \rightarrow \infty} [T_{\phi_1}^*, T_{\phi_2}]y_n$. Thus, $x \in \overline{\mbox{ran}} [T_{\phi_1}^*, T_{\phi_2}]$. This completes the proof.
\end{proof}
Using Lemma \ref{lem_com}, for pair of co-prime inner function we get
\begin{equation}\label{L_n}
\cll_0 = \phi_1 \overline{\mbox{ran}} [T_{\phi_1}^*, T_{\phi_2}].
\end{equation}
 We are now ready to state our first main result of this section. 
\begin{thm}\label{polyfinite}
Let $\phi_1, \phi_2$ be a pair of co-prime inner functions on $\D^n$ where $n>1$, then the followings are equivalent:\\ \vspace{1mm}
$(i)$~ $[P_{\phi_1},P_{\phi_2}]$ is of finite--rank,\\ \vspace{1mm}
$(ii)$~$\phi_1,\phi_2$ dependent on disjoint set of variables,\\ \vspace{1mm}
$(iii)$ ~ $[P_{\phi_1},P_{\phi_2}]=0$.
\end{thm}
\begin{proof}
$(i)\implies (ii)$ By Corollary \ref{K_finite}, $[P_{\phi_1},P_{\phi_2}]$ is of finite--rank if and only if $\cll_0 = \phi_1 \overline{\mbox{ran}} [T_{\phi_1}^*, T_{\phi_2}]$, is finite dimensional. This implies that if $[P_{\phi_1},P_{\phi_2}]$ is of finite--rank then ${\mbox{ran}} [T_{\phi_1}^*, T_{\phi_2}]$ is finite dimensional. Hence $[T_{\phi_1}^*, T_{\phi_2}]$ is of finite--rank. If $[T_{\phi_1}^*, T_{\phi_2}]$ is  of finite--rank then we have $[T_{\phi_1}^*, T_{\phi_2}]=T_{\bar\phi_1\phi_2}-T_{\phi_2}T^*_{\phi_1}=H^*_{\bar\phi_2}H_{\bar\phi_1}$, is of finite--rank. Now here by using \cite[Corollary 1]{Gu}, we get that for each $i=1,\ldots,n$, either $\bar{\phi_1}(\bm{z})$ or $\bar{\phi_2}(\bm{z})$ is analytic in $z_i$. In other words, $\phi_1$ and $\phi_2$ are dependent on disjoint set of variables.\\
$(ii)\implies (iii)$ If $\phi_1,\phi_2$ dependent on disjoint set of variables then by \cite[Corollary 2]{Gu} we get $H^*_{\bar\phi_2}H_{\bar\phi_1}=T_{\bar\phi_1\phi_2}-T_{\phi_2}T^*_{\phi_1}=0$. Hence we have
\begin{align*}
P_{\phi_2} P_{\phi_1} = T_{\phi_2} T_{\phi_2}^*T_{\phi_1} T_{\phi_1}^* &= T_{\phi_2} T_{\phi_1}T_{\phi_2}^* T_{\phi_1}^* \\
&= T_{\phi_1} T_{\phi_2}T_{\phi_1}^* T_{\phi_2}^*\\
&= T_{\phi_1} T_{\phi_1}^*T_{\phi_2} T_{\phi_2}^* \\
&= P_{\phi_1} P_{\phi_2}.
\end{align*}
$(iii)\implies (i)$ is trivial.
\end{proof}

Theorem \ref{polyfinite} have direct application to inner functions in $\mathcal{A}(\D^2)$.
\begin{cor}\label{disc_finite}
Let $\phi_1, \phi_2$ be inner functions in $\mathcal{A}(\D^2)$. Then $[P_{\phi_1}, P_{\phi_2}]$ is finite-rank if and only if $[P_{\phi_1}, P_{\phi_2}] = 0$.
\end{cor}
\begin{proof}
Following Lemma \ref{gcd}, in a similar manner, we can consider $\phi_1$ and $\phi_2$ without any common inner factor. The proof then follows by using Theorem \ref{disc_case} and Theorem \ref{polyfinite}.
\end{proof}

We will now focus on when the product $(I_{H^2(\D^n)} - P_{\phi_1})(I_{H^2(\D^n)}  - P_{\phi_2})$ becomes a compact operator.
 
\begin{thm}\label{prod_comp}
Let $\phi_1, \phi_2$ be distinct inner functions on $\D^n$, then the following are equivalent:\\ \vspace{1mm}
$(i)$~$(I_{H^2(\D^n)} - P_{\phi_1})(I_{H^2(\D^n)}  - P_{\phi_2})$ is compact,\\ \vspace{1mm}
$(ii)$~$[T_{\phi_2}^*, T_{\phi_1}]$ is compact and $\ker T_{\phi_1}^* \cap \ker T_{\phi_2}^*$ is finite-dimensional.
\end{thm}
\begin{proof}
From Corollary \ref{K finite}, we already know that the compactness of $(I_{H^2(\D^n)}  - P_{\phi_1})(I _{H^2(\D^n)} - P_{\phi_2})$ implies that $\ker T_{\phi_1}^* \cap \ker T_{\phi_2}^* = (I_{H^2(\D^n)}  - P_{\phi_1})H^2(\D^n) \cap (I_{H^2(\D^n)}  - P_{\phi_2})H^2(\D^n)$ is finite dimensional and $T = P_{\cll_0} P_{\phi_2}|_{\cll_0}$ is compact. The finite-dimensionality of $\ker T_{\phi_1}^* \cap \ker T_{\phi_2}^*$ implies $\phi_1 H^2(\D^n) \cap \phi_2 H^2(\D^n)  = \phi_1 \phi_2H^2(\D^n)$ from Theorem \ref{co-primes}. Thus, using identity (\ref{L_n}), we get
\[
\cll_0 = \overline{\mbox{ran}}[P_{\phi_1}, T_{\phi_2}].
\]
We claim that $T=  P_{\cll_0} P_{\phi_2}P_{\cll_0}$ is compact if and only if $[T_{\phi_2}^*, T_{\phi_1} ] $ is compact. Indeed,
\[
P_{\cll_0} P_{\phi_2}P_{\cll_0} \text{ compact} \implies [T_{\phi_2}^*, P_{\phi_1}]P_{\phi_2}[P_{\phi_1}, T_{\phi_2}] \text{ compact} \implies T_{\phi_2}^*[P_{\phi_1}, T_{\phi_2}] \text{ compact}.
\]
Now,
\[
T_{\phi_2}^*[P_{\phi_1}, T_{\phi_2}] = T_{\phi_2}^*P_{\phi_1}T_{\phi_2} - P_{\phi_1} = T_{\phi_2}^* \big(P_{\phi_1} - T_{\phi_2} P_{\phi_1}T_{\phi_2} ^* \big) T_{\phi_2},
\]
and
\[
\big(P_{\phi_1} - T_{\phi_2} P_{\phi_1}T_{\phi_2} ^* \big)^2 = P_{\phi_1} + T_{\phi_2} P_{\phi_1}T_{\phi_2} ^*   - P_{\phi_1}T_{\phi_2} P_{\phi_1}T_{\phi_2}^* - T_{\phi_2} P_{\phi_1}T_{\phi_2}^*P_{\phi_1} = P_{\phi_1} - T_{\phi_2} P_{\phi_1}T_{\phi_2}^*,
\]
where the last equality follows by using the invariance $ T_{\phi_2} \phi_1 H^2(\D^n) \subseteq \phi_1 H^2(\D^n)$. Therefore, the compactness of $T_{\phi_2}^*[P_{\phi_1}, T_{\phi_2}]$ implies that 
\[
 T_{\phi_2}^* \big(P_{\phi_1} - T_{\phi_2} P_{\phi_1}T_{\phi_2} ^* \big)  = [T_{\phi_2}^*, P_{\phi_1}]\text{ is compact}.
\]
This implies that $[T_{\phi_2}^*, P_{\phi_1}]T_{\phi_1} = [T_{\phi_2}^*, T_{\phi_1}]$ is compact. 
Conversely, if $[T_{\phi_2}^*, T_{\phi_1}]$ is compact, then
\begin{align*}
P_{\cll_0} P_{\phi_2}P_{\cll_0} = P_{\cll_0} P_{\phi_2}P_{\phi_1}P_{\cll_0} &= P_{\cll_0} T_{\phi_2} T_{\phi_2}^*T_{\phi_1}T_{\phi_1}^*P_{\cll_0} \\
&= P_{\cll_0} T_{\phi_2} T_{\phi_1}T_{\phi_2}^*T_{\phi_1}^*P_{\cll_0} + \text{compact}\\
&= P_{\cll_0} P_{\phi_1 \phi_2 H^2(\D^n)}P_{\cll_0} +  \text{compact}.
\end{align*}
Since, $\cll_0 \subseteq \ker T_{\phi_1 \phi_2}^*$, therefore $P_{\cll_0} P_{\phi_1 \phi_2 H^2(\D^n)}P_{\cll_0} =0$, and hence, $T = P_{\cll_0} P_{\phi_2}P_{\cll_0}$ is a compact operator. Thus, using Corollary \ref{K finite}, we get $(ii) \implies (i)$. This completes the proof.
\end{proof}
In the case of bidisc, the above characterization can be strengthened quite remarkably. The motivation comes from a recent characterization by Debnath et al. in \cite{DPS} which proved that if $(I_{H^2(\D^n)}  - P_{\phi_1})(I_{H^2(\D^n)}  - P_{\phi_2})$ is a finite-rank projection then $\phi_1$ and $\phi_2$ are separated and satisfies a finite Blaschke condition. As noted in the introduction, the following results were first observed by Zheng et al. in \cite{ZTLYZ}.
\begin{thm}\label{compact-proj1}
Let $\phi_1, \phi_2$ be distinct inner functions on $\D^2$, then the following are equivalent:\\ \vspace{1mm}
$(i)$~$(I_{H^2(\D^2)}  - P_{\phi_1})(I_{H^2(\D^2)}  - P_{\phi_2})$ is compact,\\ \vspace{1mm}
$(ii)$~$(I_{H^2(\D^2)} - P_{\phi_1})(I_{H^2(\D^2)}  - P_{\phi_2})$ is a finite-rank projection and thus, either one of the following conditions holds: 
\begin{enumerate}
\item[(a)] $\phi_1$ or $\phi_2$ is a constant function.
\item[(b)]$\phi_1$ and $\phi_2$ are separated and finite Blaschke products.
\end{enumerate}
\end{thm}
\begin{proof}
We already know from the above result that $(I_{H^2(\D^2)}  - P_{\phi_1})(I_{H^2(\D^2)}  - P_{\phi_2})$ is compact if and only if $[T_{\phi_2}^*, T_{\phi_1}]$ is compact and $\ker T_{\phi_1}^* \cap \ker T_{\phi_2}^*$ is finite-dimensional. However, using \cite[Corollary 2]{GuZ}, we can conclude that $[T_{\phi_2}^*, T_{\phi_1}]$ is compact if and only if $[T_{\phi_2}^*, T_{\phi_1}] = 0$, which further implies that $(I_{H^2(\D^2)}  - P_{\phi_1})(I_{H^2(\D^2)}  - P_{\phi_2})$ is a projection. Indeed,
\begin{align*}
(I_{H^2(\D^2)}  - P_{\phi_1})(I_{H^2(\D^2)}  - P_{\phi_2}) &= I_{H^2(\D^2)} - P_{\phi_1} - P_{\phi_2} + P_{\phi_1}P_{\phi_2} \\
&= I_{H^2(\D^2)} - P_{\phi_1} - P_{\phi_2} + T_{\phi_1}T_{\phi_1}^* T_{\phi_2}T_{\phi_2}^*\\
&= I_{H^2(\D^2)} - P_{\phi_1} - P_{\phi_2} + T_{\phi_1}T_{\phi_2}T_{\phi_1}^* T_{\phi_2}^*\\
&= I_{H^2(\D^2)} - P_{\phi_1} - P_{\phi_2} + T_{\phi_2} T_{\phi_2}^* T_{\phi_1}T_{\phi_1}^*\\
&= I_{H^2(\D^2)}  - P_{\phi_1} - P_{\phi_2} + P_{\phi_2}P_{\phi_1} \\
&=   (I_{H^2(\D^2)}  - P_{\phi_2})(I_{H^2(\D^2)}  - P_{\phi_1}).
\end{align*}
Thus, we get $(I_{H^2(\D^2)}  - P_{\phi_1})(I_{H^2(\D^2)}  - P_{\phi_2})$ is a compact projection, so it must be finite-rank. Moreover, by using \cite[Corollary 2]{Gu}, we can conclude that $[T_{\phi_1}^*, T_{\phi_2}] = 0$ implies for each $i=1,2$, either $\bar{\phi_1}(\bm{z})$ or $\bar{\phi_2}(\bm{z})$ is analytic in $z_i$. In other words, $\phi_1$ and $\phi_2$ are dependent on separate set of variables, and therefore,
\[
(I_{H^2(\D^2)}  - P_{\phi_1})(I_{H^2(\D^2)}  - P_{\phi_2}) = (I_{H^2(\D)} - P_{\phi_1}) \otimes (I_{H^2(\D)} - P_{\phi_2}) \in \clb(H^2(\D) \otimes H^2(\D)).
\]
Hence, if the product is zero, then one of the symbols is a constant function. If it is non-zero then $\phi_1$ and $\phi_2$ both must be finite-Blaschke products depending on distinct variables. The converse part is straightforward. This completes the proof.
\end{proof}

\begin{thm}\label{compact-proj2}
Let $\phi_1, \phi_2$ be inner functions on $\D^n$, where $n \geq 3$, then the following are equivalent:\\ \vspace{1mm}
$(i)$~$(I_{H^2(\D^n)}  - P_{\phi_1})(I_{H^2(\D^n)}  - P_{\phi_2})$ is finite-rank,\\ \vspace{1mm}
$(ii)$~$(I_{H^2(\D^n)}  - P_{\phi_1})(I_{H^2(\D^n)}  - P_{\phi_2})= 0$.
\end{thm}
\begin{proof}
Following the proof of Theorem \ref{prod_comp}, in a similar manner, the product $(I_{H^2(\D^n)} - P_{\phi_1})(I_{H^2(\D^n)}  - P_{\phi_2})$ is finite-rank if and only if $[T_{\phi_2}^*, T_{\phi_1}]$ is finite-rank and $\ker T_{\phi_1}^* \cap \ker T_{\phi_2}^*$ is finite-dimensional. However, using \cite[Corollary 1]{Gu}, we can again conclude that $[T_{\phi_2}^*, T_{\phi_1}] = 0$, which again by using \cite[Corollary 2]{Gu}, would imply that for each $i=1,\ldots,n$, either $\bar{\phi_1}(\bm{z})$ or $\bar{\phi_2}(\bm{z})$ is analytic in $z_i$. In other words, $\phi_1$ and $\phi_2$ are dependent on disjoint set of variables, say $\mathcal{I}, \mathcal{J}, \subseteq \{z_1, \ldots,z_n\}$, respectively such that $\mathcal{I} \cap \mathcal{J} = \emptyset$, and therefore,
\[
\begin{split}
(I_{H^2(\D^n)} - P_{\phi_1})(I_{H^2(\D^n)} - P_{\phi_2}) = (I_{H_{\mathcal{I}}^2(\D^n)}- P_{\phi_1}) \otimes (I_{H_{\mathcal{J}}^2(\D^n)} - P_{\phi_2}) \otimes I_{H_{\{z_1,\ldots,z_n\} \setminus \mathcal{I} \cup \mathcal{J}}^2(\D^n)} \\
 \in \clb(H_{\mathcal{I}}^2(\D^n) \otimes H_{\mathcal{J}}^2(\D^n) \otimes H_{\{z_1,\ldots,z_n\} \setminus \mathcal{I} \cup \mathcal{J}}^2(\D^n)).
\end{split}
\]
where, $H_{\mathcal{I}}^2(\D^n)$ denote the functions in $H^2(\D^n)$ depending only on the variables in $\mathcal{I}$. Clearly, $(I_{H^2(\D^n)}  - P_{\phi_1})(I_{H^2(\D^n)}  - P_{\phi_2}) $ cannot be finite-rank unless zero in the case when $\{z_1, \ldots,z_n\} \neq \mathcal{I} \cup \mathcal{J}$. In the other case, both $(I_{H^2(\D^n)}  - P_{\phi_1})$ and $(I _{H^2(\D^n)} - P_{\phi_2})$ must be finite-rank projections. Hence, $\ker T_{\phi_1}^*$ and $\ker T_{\phi_2}^*$ must be finite-dimensional subspaces of $H_{\mathcal{I}}^2(\D^n)$ and $H_{\mathcal{J}}^2(\D^n)$ respectively. Since, at least one of $H_{\mathcal{I}}^2(\D^n)$ and $H_{\mathcal{J}}^2(\D^n)$ are dependent on more than one variable, Ahern and Clark's result will imply either $(I _{H^2(\D^n)} - P_{\phi_1})$ or $(I _{H^2(\D^n)} - P_{\phi_2})$ is zero, and so, $
(I_{H^2(\D^n)}  - P_{\phi_1})(I_{H^2(\D^n)}  - P_{\phi_2}) =0.
$
This completes the proof.
\end{proof}

\smallskip
\noindent\textsf{Declaration of AI usage:} The ChatGPT Go version was used to conduct a literature review on rational inner functions on the polydisc. In particular, it suggested the article \cite{AJMJES}, which contains basic properties of polynomial reflections in several variables that were useful in proving (the probably well-known) Lemma 5.9. All other results in this article are proved by us without using any AI tools and take full responsibility for their correctness.
\smallskip

\noindent\textsf{Acknowledgement:} The authors thank Dr. Sudip Bhuia and Dr. Peiran Zhang for pointing out an error in an earlier version on the essential commutativity of projections. This work is supported by the Department of Science and Technology via the INSPIRE faculty fellowship DST/INSPIRE/04/2019/000769 and IIT Palakkad faculty seed grant IITP/2024/00528 of the second author.


\begin{thebibliography}{99}

\bibitem{AJMJES}
J. Agler, J.E. McCarthy, M. Stankus,
\emph{Toral algebraic sets and function theory on polydisks.}, The Journal of Geometric Analysis, 16(4), 551--562.

\bibitem{AC}
P. Ahern, D.N. Clark, 
\emph{Invariant subspaces and analytic continuation in several variables}, J. Math. Mech. 19 (1969/70), 963--969.


\bibitem{ACL}
E. Andruchow, E. Chiumiento, M.E. Di Iorio y Lucero,
\emph{Essentially commuting projections},
Journal of Functional Analysis, Volume 268, Issue 2, 2015, Pages 336--362.

\bibitem{ACG}
E. Andruchow, E. Chiumiento, G. Larotonda, 
\emph{Geometric significance of Toeplitz kernels.}
Journal of Functional Analysis, 275 (2018), no. 2, 329--355.


\bibitem{ASS}
J. Avron, R. Seiler, B. Simon, 
\emph{The index of a pair of projections}, Journal of Functional Analysis 120 (1) (1994) 220--237.



\bibitem{RV}
R.V. Bessonov,
\emph{Truncated Toeplitz operators of finite rank}, Proceedings of the American Mathematical Society, (2014): 1301--1313.

\bibitem{Beurling}
A. Beurling, 
\emph{On two problems concerning linear transformations in Hilbert space.} Acta Math. 81 (1948), 239--255.


\bibitem{BS}
A. Böttcher, I.M. Spitkovsky,
\emph{A gentle guide to the basics of two projections theory}, Linear Algebra and its Applications, Volume 432, Issue 6, 2010, Pages 1412--1459, ISSN 0024-3795.

\bibitem{CIL}
Y. Chen, K.J. Izuchi, Y.J. Lee
\emph{Kernels of Toeplitz operators on the Hardy space over the bidisk}, Journal of Functional Analysis, Volume 272, Issue 9,
2017, Pages 3869--3903.


\bibitem{BH}
A. Brown, P.R. Halmos, 
\emph{Algebraic properties of Toeplitz operators}, Crelle's journal (1964), Volume: 213, page 89--102.

\bibitem{Coburn}
LA. Coburn, 
\emph{Weyl's theorem for nonnormal operators. } Michigan Mathematical Journal 13.3 (1966): 285--288.




\bibitem{DPS}
R. Debnath, D. K. Pradhan, J. Sarkar, 
\emph{Pairs of inner projections and
two applications},  Journal of Functional Analysis. 286 (2024), no. 2, Paper No. 110216, 26 pp. 



\bibitem{GMR}
S. R. Garcia, J. Mashreghi, W. T. Ross, 
\emph{Introduction to model spaces
and their operators}, Vol. 148, Cambridge University Press, 2016.


\bibitem{Gu}
C. Gu, 
\emph{Some algebraic properties of Toeplitz and Hankel operators on polydisk.} Arch.Math. 80, 393--405 (2003). 

\bibitem{GW}
K. Guo and P. Wang,
\emph{Defect operators and Fredholmness for Toeplitz pairs with inner symbols}. J. Operator Theory 58 (2007), no. 2, 251--268.



\bibitem{GuZ}
 C. Gu, D. Zheng,
\emph{The semi-commutator of Toeplitz operators on the bidisc}. J. Operator Theory 38 (1997), no. 1, 173--193.

\bibitem{Halmos} 	
P. R. Halmos, 
\emph{Two subspaces}, Transactions of the American Mathematical Society 144 (1969) 381--389.


\bibitem{Knese}
G. Knese,
\emph{Rational inner functions on the polydisk -- a survey}, 2024. Living reference In: Alpay, D., Sabadini, I., Colombo, F. (eds) Operator Theory. Springer, 

\bibitem{ZTLYZ}
Y. Lu, R. Tian, Y. Yang, P. Zhang, C. Zu,
\emph{Compactness of products and commutators of inner projections}, https://doi.org/10.48550/arXiv.2604.22284.






\bibitem{Richman}
D.R. Richman, 
\emph{A new proof of a result about Hankel operators.}
Integral Equations Operator Theory 5 (1982), no. 6, 892--900.

\bibitem{Rudin}
W. Rudin, 
\emph{Function theory in polydiscs.}
W. A. Benjamin, Inc., New York-Amsterdam, 1969. vii+188 pp.



\end{thebibliography}
\end{document}